\documentclass[reqno]{amsart}
\usepackage{amsmath}
\usepackage{amsthm}
\usepackage{amssymb}
\usepackage{algorithmic} 
\usepackage{algorithm} 
\usepackage{graphics}
\usepackage{graphicx}
\usepackage{tikz}
\usepackage{hyperref}
\usepackage{caption}
\usepackage{mathabx}
\usepackage{cleveref}

\newtheorem{proposition}{Proposition}[section]
\newtheorem{theorem}[proposition]{Theorem}
\newtheorem{corollary}[proposition]{Corollary}
\newtheorem{lemma}[proposition]{Lemma}
\newtheorem{definition}[proposition]{Definition}

\newtheorem{example}[proposition]{Example}
\newtheorem{remark}[proposition]{Remark}
\numberwithin{equation}{section}

\newcommand{\B}{\mathbb{B}}

\newcommand{\R}{\mathbb{R}}

\newcommand{\pol}{\mathcal{P}}
\newcommand{\tpb}{\theta_{\mathcal{P},B}}
\newcommand{\np}{N_{\mathcal{P}}}

\newcommand{\x}{\bar{x}}
\newcommand{\st}{\text{ such that }}
\newcommand{\argmin}{\operatorname{argmin}}
\newcommand{\gph}{\operatorname{gph}}
\newcommand{\set}[2]{\left\{#1\,\left\vert\; #2\right.\right\}}
\newcommand{\ip}[2]{\left\langle #1,\, #2\right\rangle}

\newcommand{\inter}{\mathrm{int}\,}
\newcommand{\para}{\mathrm{par}\,}
\newcommand{\lin}{\mathrm{span}\,}
\newcommand{\cone}{\mathrm{cone}\,}
\newcommand{\Limsup}{\operatorname{Lim \, sup}}
\newcommand{\rge}{\operatorname{rge}}
\newcommand{\co}{\mathrm{conv}\,}

\newcommand{\epi}{\mathrm{epi}\;}
\newcommand{\dom}{\mathrm{dom}\;}

\begin{document}

\title[Stability results for regularized least-squares problems]{Stability results for regularized least-squares problems via  generalized Hessian expressions and monotone generalized equations}

\author{Leo Smulansky}
\address{
Department of Mathematics and Statistics\\
McGill University\\
Burnside Hall\\
805 Sherbrooke Street West\\
Montreal, Quebec H3A 0B9\\
Canada
}
\email{leo.smulansky@mail.mcgill.ca}

\address{
Department of Mathematics and Statistics\\
McGill University\\
Burnside Hall, Room 1114\\
805 Sherbrooke Street West\\
Montreal, Quebec H3A 0B9\\
Canada
}

\author{Tim Hoheisel}
\email{tim.hoheisel@mcgill.ca}

\author{Tran T.A. Nghia}
\address{
Department of Mathematics and Statistics\\
Oakland University\\
444 Mathematics and Science Center\\
Rochester, MI 48309\\
USA
}
\email{nttran@oakland.edu}

\begin{abstract}
We study perturbation and stability properties of solution mappings associated with convex regularized least-squares problems. 
We first establish an implicit function theorem for generalized equations governed by maximally monotone operators and smooth perturbations, yielding conditions for local Lipschitz continuity, directional differentiability, and semismoothness of solution mappings. We then characterize the kernel of generalized Hessians of convex functions through the subspace parallel to the subdifferential under \(\mathcal{C}^2\)-cone reducibility assumptions  on the conjugate of the regularizer, thereby replacing difficult second-order objects by tractable first-order conditions. As applications, we derive stability results for broad classes of regularized least-squares problems, including weighted polyhedral support-function regularizers and piecewise linear-quadratic penalties. Our framework unifies and extends several recent results on Lipschitz stability for LASSO-type models and monotone generalized equations.
\end{abstract}


\maketitle
\section{Introduction}

{\em Sensitivity} or  {\em stability analysis} constitutes a central theme in optimization theory. Broadly speaking, one seeks to understand how solutions of optimization problems and generalized equations vary under perturbations of the underlying data. Classical developments in this direction were driven largely by smooth nonlinear programming and variational inequalities, beginning with the foundational works of Robinson \cite{Rob 80, Rob 84, Rob 87} and   Kojima \cite{Koj 80}. We refer the reader to the standard reference by Bonnans and Shapiro \cite{BS 00} for  a very comprehensive account of the state-of-the art  (even in infinite dimensions) up to the early 2000s. Over the last two decades, however, the rapid emergence of modern applications in machine learning, statistics, imaging, compressed sensing, and signal processing has shifted attention toward nonsmooth and structured optimization models involving, e.g.,  sparsity-promoting or low-complexity regularizers. Such problems often exhibit rich variational structure while simultaneously lacking smoothness or even polyhedrality.
Among the most prominent examples are sparse (or low-rank) recovery problems such as the LASSO \cite{SaS 86, T 96}, group LASSO \cite{MiL 06}, square-root LASSO \cite{BCW 11, BCW 14}, and nuclear norm minimization \cite{CaR 09, Faz 02, RFP 10}. The vast majority of these problems are of the form 

\begin{align}\label{eq:generalLeastSquares}
    \min_{x \in \R^n}  \frac{1}{2}\|Ax-b\|^2 + \lambda r(x)
\end{align} where  $\frac{1}{2}\|\cdot\|^2$ denotes the Euclidean norm squared,  $A \in \R^{m \times n}$ is a linear (forward) operator, $b \in \R^m$ is a measurement vector,  $\lambda > 0$ is a regularization parameter,
and $r : \R^n \to \overline{\R}$ is a closed, proper,  convex regularizer that promotes structure on the solution.

\begin{remark} Without  much further effort, we could consider the (seemingly) more general problem where $\frac{1}{2}\|\cdot\|^2$ is  replaced by any convex and (locally) $\mathcal{C}^2$-smooth fidelity term. This, however, just increases notational burden, and the focus of this paper is on different classes of regularizers. 
\end{remark}



\subsection{Contributions} 
The overarching goal of this paper is to contribute to the study of the variational-analytic properties of the solution map of \eqref{eq:generalLeastSquares}, i.e., the map 
\begin{align}\label{eq:S}
  S(A,b,\lambda):= \underset{x \in \R^n}{\argmin}\;   \frac{1}{2}\|Ax-b\|^2 + \lambda r(x).
\end{align}
Like the  majority of results on solution maps to parametric optimization problems, in particular the ones pointed out in \Cref{sec:RelWork} below, we approach the study of $S$ via the fact that $S$ can be equivalently written   using the optimality conditions for \eqref{eq:generalLeastSquares}, i.e., 
\begin{align}\label{eq:S2}
S(A,b,\lambda)= \set{x\in \R^n}{0\in \frac{1}{\lambda}A^T(Ax-b)+\partial r(x)}.
\end{align}
Our study is now divided into three steps: 
\medskip

\noindent
I) We first provide a general implicit function theorem (\Cref{prop:Implicit}) for solution maps 
\[
S(p)=\set{x}{ 0\in f(p,x)+F(x)}
\]
 defined by a generalized equation governed by  maximally monotone and  set-valued part $F$ and a smooth part $f$ that is monotone in $x$. This, obviously, covers the generalized equation for $S$ in \eqref{eq:S2}. 
 This implicit function theorem provides sufficient conditions for local Lipschtiz continuity and directional differentiability of  $S$ as well {\em semismoothness*}, a more recent property coined by Gfrerer and Outrata \cite{GfO 21}, intimately connected to the more traditional notion of semismoothness popularized by Qi and Sun in their seminal paper \cite{QS 93}. The obligatory regularity condition for an implicit function theorem in our case is  simply the celebrated {\em Mordukhovich criterion}  \cite{Mor 84, Mor 92, RoW 98} for metric regularity   of the set-valued map $x\mapsto f(\bar p,x)+F(x)$ around the parameter $\bar p$ in question. This  criterion here  reads 
\begin{equation}\label{eq:KernelCQ1Intro}
\ker D_xf(\bar p,\bar x) \;\cap\; \ker D^*F(\bar x|-f(\bar p,\bar x))=\{0\},
\end{equation}
where $D^*$ is the coderivative operator (see \Cref{sec:Prelim} for details) applied to the set-valued map $F$.

\medskip

\noindent
II) When applied to \eqref{eq:S2}, the Mordukhovich criterion \eqref{eq:KernelCQ1Intro} simplifies to 
\begin{equation}\label{eq:MordDataIntro}
    \ker A \cap \ker \partial^2 r(\bar{x} | -f(\bar{p}, \bar{x})) = \{0\}, 
\end{equation}
where $ \partial^2 r:=D^*(\partial r)$ is the {\em generalized Hessian} operator (see \Cref{sec:Prelim} for details) of the convex regularizer $r$. Since this object involves a coderivative operator (applied to a subdifferential operator) it is, in general, arduous to compute. Therefore, as a second main contribution, we show in \Cref{thm:c2conepargenhess} that
for (closed, proper, convex)  functions  whose (Fenchel) conjugate  is {\em $\mathcal{C}^2$-cone reducible} the kernel of the generalized Hessian can be expressed via the {\em subspace parallel to the subdifferential}  - an object which is much more expedient to compute.
\medskip

\noindent
III) As a third main contribution, we prove stability results for the solution map \eqref{eq:S} for various choices of $r$: 
\begin{itemize}
\item[i)] We establish semismoothness (in particular local Lipschitz continuity and directional differentiability) for the case when $r^*$ is $\mathcal{C}^2$-cone reducible (\Cref{prop:stabilityc2}).
\item[ii)] We prove explicit Lipschitz bounds and expressions for the directional derivative for the case where $r$ is the composition of a polyedral support function and a linear map (\Cref{cor:polyhedraldd}).

\item[iii)] We discuss the case where $r$ is a {\em piecewise linear-quadratic (PLQ)} without appealing to $\mathcal{C}^2$-cone reducibility (\Cref{thm:plqstab}),   and prove quantitative results in the special  case where $r$ is a PLQ penalty in the sense of \cite[Example 11.18]{RoW 98} (\Cref{cor:Quant}).

\end{itemize}

\subsection{Recent related work}\label{sec:RelWork} The related work recently published decomposes into  roughly three camps based on what variational-analytical tools are being employed.  One line of research, to which this paper belongs,  employs and refines  tools from variational analysis \`a la Mordukhovich \cite{Mor 06, Mor 18},  Rockafellar and Wets \cite{RoW 98} and Dontchev and Rockafellar \cite{DoR 14} based on graphical differentiation.  Representative contributions include the work of Berk et al. \cite{BBH 22, BBH 24} for LASSO and Square-root LASSO,  where results on  Lipschitzness and directional differentiability   are obtained through second-order variational analysis in a particular instance of this paper here. Other work on LASSO-type problems exploits the explicit polyhedral structure of $\ell_1$-type regularization problems \cite{HYZ 24, MWY 25}.  More recently, Cui et. al \cite{CHNS 26} established a new perspective based on Robinson's strong regularity of suitably constructed dual problems and the geometry of $\mathcal{C}^2$-cone reducible conjugates. Their framework yielded elegant first-order characterizations for Lipschitz stability, tilt stability, and full stability of broad classes of convex regularized least-squares problems while bypassing explicit second-order calculations on the nonsmooth regularizer itself. A central feature of this approach is that the relevant generalized second-order conditions arise automatically from the structure of the dual problem and the geometry of the epigraph of the conjugate regularizer. Our work here adds to that by also providing quantitative results, and by using an implicit function theorem for the the primal problem rather than looking at strong regularity of the dual. 

Another, rather different approach was developed by Bolte et al. \cite{BLPS 21, BPS 24}  who studied differentiability properties of solution mappings to  monotone inclusion problems (thus including optimality condtions of convex problems) through the lens of  {\em path differentiability} and {\em conservative Jacobian calculus}.  

Finally, stability analysis for linear least-squares problems  in the form of \eqref{eq:generalLeastSquares} based on {\em partial smoothness} can be found in the body of work by Vaiter et al., e.g., \cite{VDF 15, VDF 17}.

\subsection{Organization} The paper is organized as follows: In \Cref{sec:Prelim} we provide the necessary background material from convex and variational analysis as well as some elementary linear-algebraic facts useful to our study. 
\Cref{sec:Implicit} contains the advertized implicit function theorem (\Cref{prop:Implicit}) along with its  technical prerequisites, including the notion of semismoothness*. In turn, \Cref{sec:Kernel} is devoted to the study of the kernel of the generalized Hessian of a convex function $g$, culminating in its characterization via the subspace parallel to the subdifferential in the case where its conjugate $g^*$ is $\mathcal{C}^2$-cone reducible (\Cref{thm:c2conepargenhess}). Finally,  the stability results for the solution map \eqref{eq:S} for different instances of the regularizer $r$ can be found in \Cref{sec:Stability}. We close with final remarks in \Cref{sec:Final}.

\medskip

\noindent
{\em Notation:} We use $\overline{\R}:=\R\cup\{\pm\infty\}$ to denote the extended real numbers. $\R_{++}$ denotes the set of positive real numbers, and similarly, $\R_{+}$ denotes the set of nonnegative real numbers and $\R_{-}$ denotes the set of nonpositive real numbers.  $\mathbb{B}_{\varepsilon}(\x)$ denotes the closed ball of radius $\varepsilon$ centered at $\x$, and we set  $\mathbb{B}:=\mathbb B_1(0)$. The distance from a point $x \in \R^d$ to a set $C \subseteq \R^d$ is denoted $d(x, C)$. All these are measured in the Euclidean norm $\|\cdot\|$ on $\R^d$.  For a function $f : \R^d \to \R$, at a point $\x$ at which $f$ is differentiable, $\nabla f(\x) \in \R^d$ is the gradient of $f$. If $f : \R^d \to \R$ is twice differentiable at $\x$, then we denote the Hessian at that point by $\nabla^2 f(\x)$. If $f : \R^n \to \R^m$ is differentiable at $\x$, then we denote its Jacobian at $\x$ by $Df(\x)$. For a subset $S$ of a real vector space, we denote the conical and linear hulls of $S$ by $\cone S$ and $\lin S$, respectively. The subspace parallel to a convex set $C$ is defined as $\lin (C - x)$, for any $x \in C$, and is denoted $\para C$. The (topological) closure of $C$ is denoted by ${\rm cl\;}C$.  The kernel and range of a matrix $A$ are denoted by $\ker A$ and $\rge A$, respectively. $\mathbb{S}^n_{+}$ denotes the set of $n \times n$ symmetric positive semidefinite matrices.  The smallest and largest singular values of $A$ are denoted $\sigma_{\min}(A)$ and $\sigma_{\max}(A)$, respectively.  For $A \in \R^{n \times n}$ symmetric, we use $\lambda_{\min}(A)$ and $\lambda_{\max}(A)$ to refer to the smallest and largest eigenvalues of $A$, respectively. 

Throughout we transition seamlessly between the notation  $\ip{x}{y}$ and $x^Ty$ for the  standard inner product between  $x,y\in \R^d$,  depending on what is more expedient in the respective situation.

\section{Preliminaries}\label{sec:Prelim}

\subsection{Tools from variational anlysis}
We present here some necessary concepts and results from variational analysis  which are needed for our study and which, in large parts,  come from the standard references \cite{DoR 14, RoW 98}. The initiated reader may skip this section and come back to it only when needed.

\subsubsection{Tools from convex analysis} 
For a  function $f : \R^d \to \overline{\R}$, its {\em epigraph} is the set $\epi f:=\set{(x,\alpha)\in \R^{d+1}}{f(x)\leq \alpha}$. Its {\em domain} is $\dom f:=\set{x\in \R^d}{f(x) < \infty}$. We call $f$ {\em proper} if $\dom f\neq \emptyset$ and $f>-\infty$. We say that $f$ is {\em closed} if $\epi f$ is closed, and {\em convex} if $\epi f$ is convex.

The \emph{(convex) subdifferential} of $f: \R^d \to \overline{\R}$ at $\bar{x}$ is \begin{align*}
    \partial f (\bar{x}) := \{v \in \R^d \ | \ f(\x)+\langle v, x - \x \rangle  \le f(x), \; \forall x \in \R^d \}.
\end{align*} 
The \emph{(Fenchel) conjugate} of   $f$ is the function $f^* : \R^d \to \overline{\R}$ defined by 
\[
f^*(y) := \sup_{x \in \R^d} \{\langle x, y \rangle - f(x)\}.
\]
As a special case of conjugacy, we note that for any (nonempty) set  $C\subseteq\R^d$, its \emph{indicator (function)} $\delta_C:\R^d\to\overline{\R}$ defined by
\[
\delta_C(x):=\begin{cases}0, & x\in C,\\ +\infty,& x\notin C
\end{cases}
\]
has the conjugate
\[
\sigma_C(y):=\delta_C^*(y)=\sup_{x\in C} \ip{y}{x},
\]
which we call the {\em support function of $C$}. In turn, if $C$ is closed and convex we have also that $\sigma_C^*=\delta_C$.
We call a proper function $f$ {\em polyhedral convex} (or {\em convex piecewise linear}) if $\epi f$ is polyhedral convex, i.e., the interection of finitely many half space. This is equivalent to $f$ having the form
\[
f(x)=\max_{i=1,\dots,r}\{a_i^Tx+\beta_i\}+\delta_\pol(x)
\]
for some (nonempty) polyhedron $\pol\subseteq\R^d$ and $(a_i,\beta_i)\in \R^d\times \R\;(i=1,\dots,r)$. In particular, compositions of polyhedral convex functions with affine maps are polyhedral convex.
We point out that \cite[ Theorem 11.14]{RoW 98} $f$ is polyhedral convex if and only if $f^*$ is.  In particular, $\sigma_C$ is polyhedral convex if and only if $C$ is polyhedral.

It is known \cite[Theorem 11.1]{RoW 98} that $f$ is closed, proper, convex if and only if $f=f^{**}(:=(f^*)^*)$ in which case we have the important relation \cite[Proposition 11.3]{RoW 98}
\begin{equation}\label{eq:SDInversion}
\partial f^* = (\partial f)^{-1},
\end{equation}
where this inversion has to be understood in the set-valued sense, to which we allude now.

\subsubsection{Tools from set-valued analysis} For a set-valued map $S : \R^{n} \rightrightarrows \R^{m}$,  we define $S^{-1}:\R^m\rightrightarrows \R^n$ via 
\[
S^{-1}(y):=\set{x\in \R^n}{y\in S(x)}.
\]
In particular, $(S^{-1})^{-1}=S$. 
The \emph{graph of $S$} is the set $\gph S := \{(x, y) \in \R^{n} \times \R^{m} \ | \ y \in S(x)\}$.   The \emph{kernel} of $S$ is $\ker S := S^{-1}(0)$.

 The \emph{outer limit} of a set-valued map 
$S : \R^{n} \rightrightarrows \R^{m}$ at $\x$ is defined as 
  \[
  \underset{x \to \x}{\Limsup} \ S(x) := \{y \in \R^{m} \ | \ \exists \{x^k\} \to \x, \ \{y^k \in S(x^k)\} \to y \}.
 \]
For $\Omega \subseteq \R^d$ and $\bar{x} \in \Omega$, we define the \emph{tangent cone} to $\Omega$ at $\x$ as 
\begin{align*}
    T_\Omega(\bar{x}) = \underset{t \downarrow 0}{\Limsup} \frac{\Omega - \bar{x}}{t}.
\end{align*} 
When $\Omega$ is convex, by \cite[Theorem 6.9]{RoW 98}, we can write $T_\Omega(\x)$ as
\begin{align*}
    T_\Omega(\bar x) = \mathrm{cl} \left \{ w \, \big| \, \exists \lambda > 0 \st \bar x + \lambda w \in \Omega \right\}.
\end{align*}
We define the \emph{regular normal cone} to $\Omega$ at $\x$ as
\begin{align*}
    \hat{N}_\Omega(\x) = \{v \in \R^d \ | \ \langle v, x - \x \rangle \le o(\lVert x - \x \rVert), \ \forall x \in \Omega\}
\end{align*} and the \emph{(limiting) normal cone} to $\Omega$ at $\x$ as
\[
N_\Omega(\bar x) = \underset{x \to \x}{\Limsup} \ \hat{N}_\Omega(x).
\]
At a point $(\bar{x}, \bar{y}) \in \gph S$, the \emph{graphical derivative} is defined as the  set-valued map $DS(\bar{x}|\bar{y}) : \R^{n} \rightrightarrows \R^{m}$
defined by 
\begin{align*}
    DS(\bar{x}|\bar{y})(w) := \{z \ | \ (w, z) \in T_{\gph S}(\bar{x}, \bar{y})\},
\end{align*} and the \emph{coderivative} is the set-valued map $D^*S(\bar{x}|\bar{y}) : \R^{m} \rightrightarrows \R^{n}$ defined by
\begin{align*}
D^*S(\bar{x}|\bar{y})(u) := \{v \ | \ (v, -u) \in N_{\gph S}(\bar{x}, \bar{y})\}.
\end{align*}
A simple observation that follows from the definition is the following.

\begin{lemma}\label{lem:CodAux} Let $F:\R^n\to \R^m$, and let  $G:\R^d\times \R^n\to \R^m$ be defined by $G(p,x)=F(x)$. Then, for $(\bar p,\bar x, \bar z)\in \gph G$, we have
\[
D^*G(\bar p, \bar x|\bar z)(y)= \{0\}\times D^*F(\bar x|\bar z)(y).
\]
\end{lemma}

\noindent
The following is a sum rule for the graphical derivative and coderivative, respectively, in the presence of continuous differentiability. 

\begin{lemma}[{\cite[Exercise 10.43]{RoW 98}}]\label{lem:Sum} Let $S=f+F$ for $f: \R^{n} \to \R^{m}$ and   $F: \R^{n} \rightrightarrows \R^{m}$. Let $(\bar x,\bar u)\in \gph S$ and assume that $f$ is continuously differentiable at $\bar x$. Then:
\begin{itemize}
    \item [(a)] $DS(\bar x|\bar u)(w)=Df(\bar x)w+DF(\bar x|\bar u-f(\bar x))(w),\quad \forall w\in \R^{n}$;
    \item[(b)]  $D^*S(\bar x|\bar u)(y)=Df(\bar x)^*y+D^*F(\bar x|\bar u-f(\bar x))(y),\quad \forall y\in \R^{m}$. 
\end{itemize}
\end{lemma}


\noindent
For a convex function $g$, at a point $(\bar{x}, \bar{v}) \in \operatorname{gph} \partial g$, we refer to the coderivative of $\partial g$ at such a point as the \emph{generalized Hessian} and denote it as follows:
\[
   \textcolor{purple}{} \partial^2 g(\bar{x} | \bar{v})(\cdot) := D^* \partial g(\bar{x} | \bar{v})(\cdot).
\]
For $f : \R^d \to \overline{\R}$ and a point $\x \in \R^d$ at which $f$ is finite, let us define the quotient
 \begin{equation}
     \Delta^2_tf(\bar x|\bar v)(w):=\dfrac{f(x+tw)-f(x)-t\langle\bar v,w\rangle}{\frac{1}{2}t^2}\quad \mbox{for}\quad t>0. 
 \end{equation}
The {\em second subderivative} of $f$ at $\bar x$ for $\bar v$ is defined by 
\[
d^2 f(\bar x|\,\bar v)(w)=\liminf\limits_{ \begin{subarray}\quad \,\ \tau\downarrow 0\\
w'\rightarrow w\end{subarray}}\Delta^2_\tau f(\bar x|\bar v)(w').
\]
The function $f$  is said to be {\it twice epi-differentiable} \cite[Definition 13.6]{RoW 98} at $\bar x\in \R^d$ for $\bar v$ if the functions  $\Delta^2_tf(\bar x|\bar v)$ {\em epi-converge} to $d^2f(\bar x|\, \bar v)$ as $t\downarrow 0$ in the sense that for every $w\in \R^d$ and any choice of  sequence $t_k\downarrow 0$ there exist $w_k\to w$ such that
$$\lim_{k\to \infty}\frac{f(\bar x+t_k w_k)-f(\bar x)-t_k \langle \bar v, w_k\rangle}{ \frac{1}{2}t_k^2}= d^2f(\bar x|\, \bar v)(w).$$
The function $f$  is twice epi-differentiable at $\bar x$, if it is twice epi-differentiable at $\bar x$ for any $v\in \partial f(\bar x)$; see \cite{MS 20} for recent developments for twice epi-differentiable functions.  

The map $F : \R^n \rightrightarrows \R^m$ is said to be \emph{proto-differentiable} at $\bar x \in \R^n$ if there exists $\bar u \in F(\x)$ such that \begin{align*}
    \frac{\gph F - (\bar x, \bar u)}{\tau} \overset{\tau \downarrow 0}{\to} \gph DF(\bar x | \bar u),
\end{align*} 
where the convergence of sets is in the sense of \cite[Chapter 4B]{RoW 98}.

A set-valued map $Q : \R^{n} \rightrightarrows \R^{m}$ is called \emph{metrically regular} at a point $(\bar{x}, \bar{y}) \in \gph Q$ if
there exists a neighborhood $V$ of $\bar{x}$, a neighborhood $W$ of $\bar{y}$, and $\kappa \ge 0$
such that
\begin{align*}
    d(x, Q^{-1}(y)) \le \kappa d(y, Q(x)), \ \forall x \in V, y \in W.
\end{align*}
\noindent
For $S : \R^n \rightrightarrows \R^m$  and $(\x, \bar y) \in \gph S$, then we say $S$ has a \emph{localization} at $\x$ for $\bar y$ if there are neighborhoods $V$ of $\x$ and $W$ of $\bar y$ such that $x \in V \mapsto S(x) \cap W$ is single-valued. We say $S$ has a \emph{Lipschitz continuous localization} at $\bar x$ for $\bar y$ if this localization is Lipschitz continuous. 

$Q : \R^n \rightrightarrows \R^m$ is called \emph{strongly metrically regular} at $(\bar{x}, \bar{y}) \in \gph Q$ if $Q^{-1}$ has a Lipschitz continuous localization around $\bar{y}$ for $\bar{x}$.

\begin{lemma}[Single-valuedness of localization from convex-valuedness]\label{lem:localizationAndConvexity}
    Let $S : \R^n \rightrightarrows \R^m$ be convex-valued \footnote{i.e. $S(x)$ is a convex set for all $x \in \R^n$.}, let $(\bar x, \bar y) \in \gph S$, and suppose $S$ has a Lipschitz continuous localization at $\bar x$ for $\bar y$.
    Then, $S$ is locally (single-valued and) Lipschitz continuous at $\bar x$.
\end{lemma}

\begin{proof}
    Let $V$ be the associated neighborhood of $\x$, and $W$ be the associated neighborhood of $\bar y$ such that $x \in V \mapsto S(x) \cap W$ is single-valued. Without loss of generality (by shrinking $V$ is necessary), we may assume $S$ is convex-valued on $V$. Fix $x \in V$ and let $y \in S(x)$. Let $s : V \to W$ be a localization of $S$ which is Lipschitz continuous.  By convexity of $S(x)$, for all $\lambda \in (0, 1)$, we have $s(x) + \lambda(y - s(x)) \in S(x)$. For $\lambda$ sufficiently small, $s(x) + \lambda(y - s(x)) \in S(x) \cap W = \{s(x)\}$, so $s(x) + \lambda(y - s(x)) = s(x) \implies s(x) = y$. Therefore, $S(x) = \{s(x)\}$ for $x \in V$, which establishes that $S$ is in fact locally single-valued and Lipschitz continuous at $\bar x$.
\end{proof}

A set-valued map $Q : \R^d \rightrightarrows \R^d$ is called \emph{monotone} if \begin{align*}
    \langle v_1 - v_0, x_1 - x_0 \rangle \ge 0, \ \forall v_0 \in Q(x_0), \ v_1 \in Q(x_1),
\end{align*} and is called \emph{maximally monotone} if no enlargement of its graph is possible in $\R^d \times \R^d$ without destroying monotonicity.

 The following result from \cite[Theorem 3G.5]{DoR 14} will be one of the central tools we use to develop the implicit function theorem.

\begin{proposition}\label{thm:monotonicityAndMetricRegularity}
    If $Q : \R^d \rightrightarrows \R^d$ is monotone and metrically regular at $(\x, \bar y) \in \gph S$, then in fact $Q$ is strongly metrically regular at $(\x, \bar y).$
\end{proposition}

Multiple times in this paper, we will employ (explicitly or implicitly) the following result, which is now known as the {\em Mordukhovich criterion}, see, e.g.,  \cite[Theorem 9.40]{RoW 98}.
\begin{theorem}[Mordukhovich criterion]\label{thm:coderivativeCriterion}
    Let $Q : \R^n \rightrightarrows \R^m$ with $(\x, \bar y) \in \gph Q$ and suppose $\gph Q$ is  closed. Then, $Q$ is metrically regular at $(\x, \bar y)$ if and only if \begin{align*}
        \ker D^*Q(\bar{x} | \bar{y}) = \{0\},
    \end{align*} or, equivalently, if \begin{align*}
        0 \in D^*Q(\x | \bar y)(w) \implies w = 0.
    \end{align*}
\end{theorem}

\subsection{Linear algebra results}

We present here some auxiliary results from linear algebra which are useful to our study. 

The first two  lemma are  more generally relevant and certainly known, but we provide  proofs for completeness.
\begin{lemma}\label{lem:Sub} Let $M\in \R^{l\times n}$, $U,V\subseteq\R^l$ subspaces. Then:
\begin{itemize}
    \item[(a)] $M^{-1}(V^\perp)=(M^T(V))^\perp$;
    \item[(b)] $\left(M^{-1}(U)\right)^\perp=M^T(U^\perp)$.
\end{itemize}
\end{lemma}
\begin{proof}(a) We have 
\begin{eqnarray*}
   M^{-1}(V^\perp) & =& \set{x}{Mx\in V^\perp}\\
   & = & \set{x}{\ip{x}{M^Tv}=0\; , \, \forall v\in V}\\
   & = & \set{x}{\ip{x}{s}=0\;, \, \forall s\in M^T(V)}\\
   & = & (M^T(V))^\perp.
\end{eqnarray*}
(b)    We have 
\begin{eqnarray*}
   (M^{-1}(U))^\perp & = & (M^{-1}\left((U^\perp)^\perp\right))^\perp\\
   &\overset{(a)}{=} & \left(\left(M^T(U^\perp)\right)^\perp\right)^\perp\\
   & = & M^T(U^\perp).
\end{eqnarray*}
\end{proof}

\begin{lemma}\label{lem:Hinvert}
    Let $C, D \in \R^{n \times n}$ be symmetric positive semidefinite matrices.
    Then, $I + CD$ is invertible.
\end{lemma}

\begin{proof}
    Take $v \in \ker (I + CD)$, i.e., $v + CDv = 0$.
    Left-multiplying this expression by $(Dv)^T$, we get \begin{align*}
        v^TDv + v^TDCDv = 0.
    \end{align*} By positive semidefiniteness of $C$ and $D$,
    both $v^TDv=0$ and $v^TDCDv=0$.  However, $v^TDv = 0$
    implies that $Dv=0$, thus $0=v + CDv = v$, so we
    have $v = 0$. This proves the statement.
\end{proof}

\noindent
The last lemma is more tailored to our specific needs and will come into play in our quantitative analysis of the solution maps in \Cref{cor:Quant}.

\begin{lemma}\label{lem:AHUinvert}
    Let $C \in \R^{n \times n}$ be symmetric positive semidefinite, $A \in \R^{m \times n}$, $E \subseteq \R^n$
    be a subspace such that $\ker A \cap E = \{0\}$. Let $H := I + CA^TA$ and let $U$
    have columns $\begin{bmatrix}
        u_1 \cdots u_r
    \end{bmatrix}$ which form an orthonormal basis of $H^{-1}(E)$.
    Then the following hold:
    \begin{itemize}
        \item[(a)] $\ker AHU = \{0\}$;
        \item[(b)] $H(\ker A) \subseteq \ker A$;
        \item[(c)] $\ker AU = \{0\}$;
        \item[(d)] $(AHU)^T AU$ is symmetric positive definite.
    \end{itemize}
\end{lemma}

\begin{proof}
    (a) $\ker A \cap E = \{0\}$ gives us that $H^{-1} (\ker A) \cap H^{-1}(E) = \ker H=\{0\}$,  where the latter uses \Cref{lem:Hinvert}. Observe that  
    $H^{-1}(\ker A)=\ker AH$,   and that $\rge U =H^{-1}(E)$ by construction.  Thus,  $\ker AH \cap \rge U = \{0\}$, which implies $\ker AHU = \{0\}$.

    (b) Let $y \in \ker A$. Then, $Hy = y + C A^TAy$, but $Ay = 0$, so $Hy = y$. Thus, $Hy \in \ker A$,
    which tells us $H(\ker A) \subseteq \ker A$.

    (c) Let $y \in \ker AU$. Then, $Uy \in \ker A$, so by (c), $HUy \in \ker A$.
    This gives us $y \in \ker AHU$, hence $y=0$ by (a). Thus, $\ker AU = \{0\}$.

(d) 
Note that
 \begin{align*}
   (AHU)^T (AU) = U^T \left(I + A^TA C\right) A^TAU = (AU)^T(AU) + (AU)^TC(AU).
 \end{align*} 
Thus, $(AHU)^T (AU)$ is symmetric and  by (c),  $(AU)^T(AU)$ is positive definite which then gives the desired statement.
\end{proof}

\section{An Implicit Function Theorem for Monotone Maps}\label{sec:Implicit}

\noindent
In this section we present an implicit function theorem for monotone (set-valued) maps that will be a main tool for establishing the (variational-)analytic properties (such as directional differentiability, local Lipschitzness) of the solution maps in which we are interested. To this end, we need to extend our variational analysis toolkit by a notion of {\em semismoothness} for set-valued maps which was popularized by Gfrerer and Outrata \cite{GfO 21}.


\begin{definition}[Semismoothness*]
	The set $A\subseteq \R^d$ is {\em semismooth*} at $\bar x\in A$ if  
	\[
	\ip{x^*}{u}=0\quad \forall u\in \textcolor{purple}{\R^d},\;  x^*\in N_A(\bar x;u)
	\]
    where $N_A(\x ; u)$ is the \emph{directional normal cone} to $A$ at $\x$ in direction $u$ (see \cite{GfO 21} for definition.)
	The map $S:\R^{n} \rightrightarrows \R^{m}$ is {\em semismooth*} at $(\bar x,\bar y)\in \gph S$ if $\gph S$ is semismooth* at $(\bar x,\bar y)$, i.e.,
	\[
	\ip{u}{u^*}=\ip{v}{v^*}\quad \forall (u,v)\in \R^{n} \times \R^{m},\; (v^*,u^*)\in\gph D^*S((\bar x,\bar u); (u,v))
	\] where $D^*S((\x, \bar u) ; (u, v))$ is the \emph{directional coderivative} (also in \cite{GfO 21}.)
\end{definition}

\noindent
Recall also the classical notion of semismoothness, initally introduced by Mifflin in \cite{Mif 77}, and later generalized to  the vector-valued case by Qi and Sun \cite{QS 93}.

\begin{definition}[Semismoothness]
    A function $F : \R^{n} \to \R^{m}$ is called \emph{semismooth} at $\x$ if $F$ is locally Lipschitz at $\x$, and for any $h \in \R^{n}$, the following limit exists:
    \begin{align*}
        \lim_{
        \substack{V \in \partial_CF(x + t h') \\
        h' \to h, \ t \downarrow 0
        }} Vh'
    \end{align*} where $\partial_C F$ is  {\em Clarke's generalized Jacobian} operator \footnote{$\partial_C F(x):=\co \set{V}{\exists \{x_k\in D_F \}\to x: DF(x_k)\to V}$ where $D_F$ is the set of differentiability of $F$.} \cite{Cla 83}.
\end{definition}

\noindent
The significance of semismoothness* for our study is highlighted by the following result from \cite[Theorem 3.8]{GfO 21}, which illustrates the intimate connection between the notions of semismoothness and semismoothness*.

\begin{lemma}[Semismooth vs. semismooth*]\label{lem:Semi} Let $F:D\subseteq \R^{n} \to \R^{m}$ be locally Lipschitz at $x\in\inter D$.  Then the following are equivalent:
\begin{enumerate}
\item[(i)] $F$ is semismooth at $x$; 
\item[(ii)] $F$ is semismooth* and directionally differentiable at $x$.
\end{enumerate}
\end{lemma}
\noindent
We emphasize that, in particular, a semismooth function is always locally Lipschitz (by definition) and  directionally differentiable.

Recall the following result from \cite[Proposition 2]{FGH 22} which is a preimage rule for semismooth* sets.

\begin{proposition}[Metric regularity and semismoothness*]\label{prop:Semi} Let $H: \R^{n} \to \R^{m}$ be continuously differentiable at $\bar z$,  let $Q\subseteq\R^{m}$ be semismooth* (as a set) at $H(\bar z)$ and let $\Psi: \R^{n} \rightrightarrows \R^{m},\; \Psi(z):=H(z)-Q$ be metrically  
(sub)regular at $(\bar z,0)$. Then $ H^{-1}(Q)$ is semismooth* at $\bar z$ (as a set).
\end{proposition}

\noindent
As a consequence we obtain  a result about semismoothness* of implicit functions.

\begin{corollary}\label{cor:SemiImplicit} Let $f: \R^m \times \R^n \to \R^n$ be continuously differentiable at $(\bar p,\bar x)$, and let $F:\R^n \rightrightarrows \R^n$ be semismooth* at $(\bar x,-f(\bar p,\bar x))$ (in particular, $0\in f(\bar p,\bar x)+ F(\bar x)$). Define $S:\R^m \rightrightarrows \R^n$ via 
\[
S(p)=\set{x\in \R^n}{0\in f(p,x)+F(x)}.
\]
Then $S$ is semismooth* at $(\bar p,\bar x)$ provided that the set-valued mapping $\Psi:\R^m\times \R^n \rightrightarrows \R^n \times \R^n$ defined by  $\Psi(p,x)=(x,-f(p,x))-\gph F$ is metrically regular at $(\bar p,\bar x,0)$ \footnote{In fact, we only require $\Psi$ to be \emph{metrically subregular} at the point in question (see \cite[Chapter 3H]{DoR 14})}.  This is equivalent to the following implication being satisfied:
\begin{equation}\label{eq:MordSemi}
\left.\begin{array}{rcl}
-(v,w)&\in &N_{\gph F}(\bar x,-f(\bar p,\bar x)),\\
0 & = & D_pf(\bar p,\bar x)^*w,\\
v & = & D_xf(\bar p,\bar x)^*w
\end{array}\right\}  \quad \Longrightarrow\quad (v,w)=(0,0).
\end{equation}

\end{corollary}

\begin{proof} Define the function $H:\R^m \times \R^n \to \R^n \times \R^n,\; H(p,x)=(x,-f(p,x))$, which is continuously differentiable at $(\bar p,\bar x)$. We note that 
\begin{eqnarray*}
(p,x)\in \gph S & \Longleftrightarrow & (x,-f(p,x)) \in \gph F\\
& \Longleftrightarrow & (p,x) \in H^{-1}(\gph F),
\end{eqnarray*}
i.e. $\gph S=H^{-1}(\gph F)$. 
Since $F$ is assumed  semismooth* at $(\bar x,-f(\bar p,\bar x))$, by definition, $\gph F$ is semismooth* (as a set) at $(\bar p,\bar x)$. Consequently, \Cref{prop:Semi} (with $Q=\gph F$) yields that $\gph S=H^{-1}(\gph F)$ is semismooth* at $(\bar p,\bar x,0)$ if $\Psi$ is metrically subregular at $(\bar p,\bar x,0)$. 

To prove the remainder,  recall that, by the Mordukhovich criterion (\Cref{thm:coderivativeCriterion}),   $\Psi$ is metrically regular at $(\bar p,\bar x,0)$ if and only if the following implication holds: 
\begin{equation}
0\in D^*\Psi(\bar p,\bar x|0)(v,w) \quad \Longrightarrow \quad (v,w)=(0,0).
\end{equation}
Now, observe that, by the coderivative sum rule \Cref{lem:Sum}, we have 
\begin{eqnarray*}
D^*\Psi(\bar p,\bar x|0)(v,w) & = & DH(\bar p,\bar x)^*(v,w)+\begin{cases} \{0\}\times \{0\}, &-(v,w) \in N_{\gph F}(H(\bar p,\bar x)),\\ \emptyset, & \text{else}.\end{cases}
\end{eqnarray*}
Since $DH(\bar p,\bar x)^*(v,w)=[-D_pf(\bar p,\bar x)^*w, \; v-D_xf(\bar p,\bar x)^*w]$, we find that  $0\in D^*\Psi(\bar p,\bar x|0)(v,w)$ holds if and only if

\begin{equation}
\begin{array}{rcl}
-(v,w)&\in &N_{\gph F}(\bar x,-f(\bar p,\bar x)),\\
0 & = & D_pf(\bar p,\bar x)^*w,\\
v & = & D_xf(\bar p,\bar x)^*w.
\end{array}
\end{equation}
This concludes the proof.
\end{proof}

\noindent 
 We now present the  advertized implicit function theorem which constitutes a polished and substantially amended version of   \cite[Proposition 4.11]{BBH 22}.



\begin{proposition}\label{prop:Implicit} Let $(\bar p,\bar x)\in \R^d\times \R^n$, let  $F:\mathbb \R^n\rightrightarrows \R^n$ be maximally  monotone, and  let $f:\R^d\times \R^n\to \R^n$  be continuously differentiable at $(\bar p, \bar x)$ such that $f(p,\cdot)$ is  monotone for any $p$ near $\bar p$. Define $S:\R^d\rightrightarrows\R^n$ by 
\[
S(p)=\set{x\in \R^n}{0\in f(p,x)+F(x)}, \quad \forall p \in \mathbb{R}^d.
\]
Assume that  $(\bar p,\bar x)\in \gph S$  (i.e., $0\in f(\bar p,\bar x)+F(\bar x)$) such that
\begin{equation}\label{eq:MordCrit}
   0\in D_x f(\bar p,\bar x)^*w+D^*F(\bar x\mid -f(\bar p,\bar x))(w)\quad\Longrightarrow \quad w=0,
\end{equation}
or, equivalently, 
\begin{equation}\label{eq:KernelCQ}
\ker D_xf(\bar p,\bar x) \cap \ker D^*F(\bar x|-f(\bar p,\bar x))=\{0\}.
\end{equation}
Then the following hold:
\begin{itemize}
    \item[(a)]  $Q:=f(\bar p,\cdot)+F:\R^n\rightrightarrows \R^n$ is  {\em strongly metrically regular} at $(\bar x,0) \in \gph Q$ (and maximally monotone).
    \item[(b)]   $S$ is (single-valued) locally Lipschitz at $\bar p$. 
    \item[(c)] If $F$ is {\em proto-differentiable}  at $(\bar x, -f(\bar p,\bar x))$, then  the graphical derivative $DS(\bar p|\bar x)$ is single-valued and locally  Lipschitz with 
    \begin{equation}
    \label{eq:graphical_derivative}
     DS(\bar p)(q)=\set{w \in \mathbb{R}^n}{0\in DG(\bar p,\bar x|0)(q,w)}, \quad \forall q \in \mathbb{R}^d,
    \end{equation}
    for $G(p,x):=f(p,x)+F(x)$.
    In particular, $S$ is directionally differentiable\footnote{i.e., $S'(\bar x;d):=\lim_{t\downarrow 0}\frac{S(\bar x+td)}{t}$ exists for all $d\in \R^d$.} at $\bar p$ with (locally Lipschitz) directional derivative
    \[
    S'(\bar p;\cdot)=DS(\bar p)(\cdot).
    \]
    In addition,  $S$ is locally Lipschitz at $\bar p$ with  modulus
    \[
    L\leq\limsup_{p\to \bar p} \max_{\|q\|\leq 1} \|DS(p)(q)\|.
    \]
    \item[(d)]  If $F$ is semismooth* at $\bar x$, then $S$ is semismooth* at $\bar p$. Thus, if $F$ is also proto-differentiable at $\bar x$, then $S$ is semismooth at $\bar p$.

\end{itemize}
\end{proposition}

\begin{proof}  We first show that \eqref{eq:MordCrit} and \eqref{eq:KernelCQ} are equivalent. To this end, observe that
\[
0\in D_x f(\bar p,\bar x)^*w+D^*F(\bar x\mid -f(\bar p,\bar x))(w)
\]
is, equivalent to, 
\[
- D_x f(\bar p,\bar x)^*w\in D^*F(\bar x\mid -f(\bar p,\bar x))(w).
\]
By the monotonicity properties of $f(\bar p, \cdot)$ and $F$ combined with \cite[Theorem~2.1]{PR 98}, this implies
\[
0\leq-\ip{ D_x f(\bar p,\bar x)^*w}{w}\leq 0,
\]
and  hence, $w\in \ker  D_x f(\bar p,\bar x)$. Therefore, we have
\[
0\in D^*F(\bar x\mid -f(\bar p,\bar x))(w),
\]
or, equivalently, 
\[
w\in \ker  D^*F(\bar x| -f(\bar p,\bar x)).
\]
Therefore,  \eqref{eq:MordCrit} implies \eqref{eq:KernelCQ}. The  reverse implication follows  from  the same observations, yet starting at \eqref{eq:KernelCQ}.
\smallskip

\noindent
(a) Because $Q$ is a sum of maximally monotone maps\footnote{Note that a single-valued continuous monotone map is maximally monotone, see, e.g., \cite[Example 12.7]{RoW 98}. } and $\dom f(\bar p,\cdot)=\R^n$, it is maximally monotone \cite[Corollary 12.44]{RoW 98}.  Hence $Q$ has closed graph and is convex-valued, see., e.g., \cite[Exercise 12.8]{RoW 98}. Since $Q$ has closed graph with $(\bar{x}, 0) \in \gph Q$ and
    $\ker D^{*}Q(\bar{x}|0) = \{0\}$, it holds that $Q$ is metrically
    regular at $(\bar{x}, 0)$ by \Cref{thm:coderivativeCriterion}.     
Together with the monotonicity property,  \Cref{thm:monotonicityAndMetricRegularity} immediately gives
    that $Q$ is \emph{strongly metrically regular} there. 
\smallskip

\noindent
(b) We first prove that  that the function $h:=f(\bar p,\cdot)$ is a {\em strict estimator} of $f$ with respect to $x$ {\em uniformly} in $p$ at $(\bar p,\bar x)$ with constant $\mu=0$ in the sense of \cite[p.~47]{DoR 14} (cf. also \cite[p.~36]{DoR 14}): To this end, observe that with $e(p,x):=f(p,x)-h(x)$, we have
\begin{eqnarray*}
\lefteqn{\|e(p,x')-e(p,x)\|}\\
& =  &  \|f(p,x')-f(p,x) + f(\bar p,x)-f(\bar p,x')\|\\  
& \leq & \|f(p,x')-f(p,x)-D_x f(p,x)(x'-x)\|+\|f(\bar p,x')-f(\bar p,x)-D_x f(\bar p,x)(x'-x)\|\\
 & & + \|D_x f(\bar p,x)-D_x f(p,x)\|\cdot\|x-x'\|.
\end{eqnarray*}
Therefore, we find that 
 \[
 \lim_{\substack{x, x' \to \x \\ p \to \bar p}} \frac{e(p,x')-e(p,x)}{\|x-x'\|}=0.
 \]
By (a), the mapping $f(\bar p,\cdot)+F$ is strongly metrically regular at $(\bar x,0)$. By Theorem~\cite[Theorem~3G.4]{DoR 14}, the solution mapping $S$ has a Lipschitz continuous single-valued localization around $\bar p$  for $\bar x$. 
We now find that $S(p)=G(p,\cdot)^{-1}(0)$ is convex-valued for $p$ near $\bar p$, by maximal monotonicity of $G(p,\cdot)$ (see the arguments used for $Q$ above), see \cite[Exercise 12.8]{RoW 98}. By \Cref{lem:localizationAndConvexity}, $S$ is thus locally single-valued and Lipschitz  which proves (b).


\smallskip

\noindent
(c) Realize, by~\cite[Lemma 4]{FGH 22}, that $G$ is proto-differentiable at
$((\bar p, \bar x),0)$. Therefore, by \cite[Theorem 9.56 (c)]{RoW 98}, we get that $S$ is \emph{semidifferentiable}
(cf., \cite[Chapter 8H]{RoW 98}) at $\bar p$ and
that~\cref{eq:graphical_derivative} holds. The claim about single-valuedness and Lipschitzness of $DS(\bar p)$ also follows from~\cite[Theorem 9.56 (c)]{RoW 98} realizing that strong metric regularity from (a) implies the strict graphical derivative condition that is needed in said reference. The fact that the semiderivative is the directional derivative is due to local Lipschitzness of $S$ near $\bar p$. The formula for the estimate of the Lipschitz constant comes from \cite[Theorem 4B.2]{DoR 14}.

\smallskip

\noindent
(d) By $(a)$, $Q=f(\bar p,\cdot)+F$ is (strongly) metrically regular at $(\bar x,0)$, i.e. 
\[
0\in D^*Q(\bar x|0)(w)\quad\Longrightarrow\quad w=0.
\]
By  \Cref{lem:Sum} and the definition of the coderivative, this is equivalent to
\[
-(D_x f(\bar p,\bar x)^*w,w)\in N_{\gph F}(\bar x,-f(\bar p,\bar x))\quad\Longrightarrow\quad w=0.
\]
Hence, given $(v,w)\in N_{\gph F}(\bar x,-f(\bar p,\bar x))$ such that $v=(D_x f(\bar p,\bar x)^*w$, it follows that 
$w=0$, and hence $v=0$. Consequently, $D_pf(\bar p,\bar x)^*w=0$, and altogether the implication \eqref{eq:MordSemi} holds. Therefore, by \Cref{cor:SemiImplicit}, $S$ is semismooth* at $\bar p$.
 But by (b), $S$ is locally Lipschitz  at $\bar p$, hence \Cref{lem:Semi} gives the desired result. If $F$ is proto-differentiable at $(\bar x, -f(\bar p, \x))$, then $S$ is directionally differentiable at $\bar p$, so by \Cref{lem:Semi}, $S$ is semismooth at $\bar p$.

\end{proof}


\section{The kernel of the generalized Hessian and the role $\mathcal{C}^2$-cone reducibility}\label{sec:Kernel}

\noindent
We want to  use the implicit function theorem from \Cref{prop:Implicit} to study the solution map  \eqref{eq:S}.  This can be done by letting 
\begin{equation}\label{eq:Data}
p:=(A,b,\lambda), \quad f(p,x) := \frac{1}{\lambda} A^T(Ax - b), \quad F := \partial r,
\end{equation}
and realizing that now, by convexity, the optimal solution map in \eqref{eq:S}  can be written simply as 
\[
S(p)=\set{x\in \R^n}{ 0\in f(p,x)+F(x)}
\]
which is exactly the format in \Cref{prop:Implicit}.
The central condition to deploy this result is the Mordukhovich criterion from \eqref{eq:MordCrit} or, equivalently, \eqref{eq:MordCrit} which  in the case defined by \eqref{eq:Data}  reduces to 
\begin{align*}
    \ker A \cap D^* \partial r^* (-f(\bar{p}, \bar{x}))(0) = \{0\},
\end{align*} 
which, by \eqref{eq:SDInversion} and \cite[Eq. 8(19)]{RoW 98},  can be written as 
\begin{equation}\label{eq:MordData}
    \ker A \cap \ker \partial^2 r(\bar{x} | -f(\bar{p}, \bar{x})) = \{0\}.
\end{equation}

\noindent
It is therefore paramount to understand the kernel of the generalized Hessian of a closed, proper, convex function.  
This is where the notion of $\mathcal{C}^2$-cone reducibility \cite[Definition 3.135]{BS 00} in the following sense comes into play:

\begin{definition}[$\mathcal{C}^2$-cone reducible functions]
    A closed, convex set $\Theta \subseteq \R^m$ is said to be $\mathcal{C}^2$-cone reducible
    at $\bar{v}\in \Theta$ if there exists a neighborhood $V$ of $\bar{v}$, a pointed\footnote{A cone $K$ is called poined if $K\cap (-K)=\{0\}$.}, closed, convex cone $K \subseteq \R^{\ell}$
    and a $\mathcal{C}^2$ mapping $h : V \to \R^{\ell}$ such that $h(\bar{v}) = 0$, $D h(\bar{v})$ is surjective, and \begin{align*}
        \Theta \cap V = \{v \in V \ | \ h(v) \in K\}. 
    \end{align*} A closed, proper, convex function $f : \R^n \to \overline{\R}$ is called $\mathcal{C}^2$-cone reducible at $\bar{u} \in \operatorname{dom}f$
    if the set $\operatorname{epi} f$ is $\mathcal{C}^2$-cone reducible at $(\bar{u}, f(\bar{u}))$. We say $f$ is $\mathcal{C}^2$-cone reducible
    if it is $\mathcal{C}^2$-cone reducible at every point of its domain.
\end{definition}

Many important sets in optimization are $\mathcal{C}^2$-cone reducible such as convex polyhedral sets, the cone of positive semidefinite matrices, and the second-order cone   \cite{BS 00, Sh03,BR05}. Moreover,   convex piecewise-linear functions and many typical {\em spectral functions} are $\mathcal{C}^2$-cone reducible \cite{BS 00,CDZ17}; see also \cite[Remark~2]{CHNS 26} for several other classes of $\mathcal{C}^2$-cone reducible functions.

The central result of this section is obtaining the formula for the kernel of the generalized Hessian of {\em conjugate $\mathcal{C}^2$-cone reducible functions}, i.e., functions whose conjugates are $\mathcal C^2$-cone reducible. In order to do so, we need the landmark definition of tilt stability \cite{PR 98} introduced by Poliquin and Rockafellar and several related results.

\begin{definition}[Tilt stability]
 Given a function $\varphi : \R^n \to \overline{\R}$, we say that a point $\x \in \operatorname{dom} \varphi$
    is a \emph{tilt-stable minimizer} of $\varphi$ if there exists $\gamma > 0$ such that the mapping
    \begin{align*}
        M_{\gamma} : v \mapsto \argmin \{\varphi(x) -  \ip{v}{x} \ | \ x \in \mathbb{B}_{\gamma}(\x)\}
    \end{align*}
is single-valued and Lipschitz continuous on some neighborhood of $0\in \R^n$ with $M_\gamma(0)=\x$.
\end{definition}

\noindent
We recall \cite[Theorem 1.3]{PR 98}, with some modifications for the convex case.

\begin{theorem}[Characterization of tilt-stability via the generalized Hessian]\label{thm:tiltstab}
    Let $\varphi:\R^n \to \overline{\R}$ be a proper, lsc, convex function. Suppose $\bar{x}$ is a minimizer of $\varphi$ with $0 \in \partial \varphi (\bar{x})$.
    Then, the following are equivalent:
    \begin{itemize}
        \item[(i)] The point $\bar{x}$ is a tilt-stable minimizer of $\varphi$.
        \item[(ii)] The generalized Hessian $\partial^2 \varphi (\bar{x} | 0)$ is positive definite, in the sense that
        \begin{align}\label{eq:PHes}
          \ip{z}{x}  > 0 \text{ whenever } z \in \partial^2 \varphi(\bar{x} | 0)(w), \ w \not = 0.
        \end{align}
    \end{itemize}
\end{theorem}

\noindent
A large portion of the following result is taken from \cite{CHNS 26}.

\begin{theorem}[Characterization of tilt stability for convex problems]\label{thm:tiltstab2}
    Let $f : \R^n \to \overline{\R}$ be a twice continuously differentiable
    convex function and let $g : \R^n \to \overline{\R}$ be a proper lsc convex function.
    Suppose that $\bar{x}\in \argmin f+g$. Then  $\bar{x}$ is a tilt-stable minimizer of $f + g$
    if and only if 
    \begin{align}\label{eq:Nes}
        \ker \nabla^2 f(\bar{x}) \cap \ker \partial^2 g(\bar{x} | \bar{v}) = \{0\} \text{ where } \bar{v} = -\nabla f(\bar{x}).
    \end{align} Moreover, if $\bar{x}$ is a tilt-stable minimizer of $f + g$, we have 
    \begin{align}\label{eq:Kpar}
        \ker \nabla^2 f(\bar{x}) \cap \operatorname{par} \partial g^*(\bar{v}) = \{0\}.
    \end{align}
    If, additionally, $g^*$ is $\mathcal{C}^2$-cone reducible at $\bar{v}$, then \eqref{eq:Kpar} is also sufficient for tilt stability of $f + g$ at $\bar{x}$.
\end{theorem}

\begin{proof}
    Define $\varphi := f + g$ and note that $0\in\partial \varphi(\x)$, as $\x$ is a minimizer of $\varphi$. By the second-order subdifferential sum rule  or \Cref{lem:Sum},  we have that 
    \begin{align*}
        \partial^2 \varphi(\bar{x} | 0)(w) = \nabla^2 f(\bar{x})w + \partial^2 g(\bar{x} | \bar v)(w) \ \text{ for all } w \in \R^n.
    \end{align*} 
    If  $\x$ is a tilt-stable minimizer of $\varphi$, we claim \eqref{eq:Nes}. Indeed, pick any $w \in \ker \nabla^2 f(\bar{x}) \cap \ker \partial^2 g(\bar{x} | \bar{v})$. The above equation leads us to
     $0\in \partial^2 \varphi (\bar{x} | 0)(w)$.
    By ~\Cref{thm:tiltstab}, particularly \eqref{eq:PHes}, $w = 0$. This verifies \eqref{eq:Nes}. On the other hand, if \eqref{eq:Nes}  holds, applying ~\Cref{prop:Implicit}(a) with $f(p,x)=\nabla f(x)$ and $F=\partial g(x)$ tells us that $\partial \varphi=\nabla f(x)+\partial g(x)$ is strongly metrically regular at $(\x,0)\in \gph \partial \varphi$. Since $\varphi$ is convex, this ensures $\x$ is a tilt-stable minimizer of $\varphi$; see e.g.,  \cite{AG08}.
    
    The necessary and sufficient condition for tilt stability in \eqref{eq:Kpar} follows from \cite[Theorem 4.3 and Theorem 4.5]{CHNS 26}, when the conjugate function $g^*$ is $\mathcal{C}^2$-cone reducible at $\bar v$.
\end{proof}

\noindent
 Part of the following lemma is from \cite[Proposition~3.1]{TW24}. 
\begin{lemma}\label{lem:protodiff}
    Suppose $h: \R^n \to \overline{\R}$ is $\mathcal{C}^2$-cone reducible at $\bar {u}\in {\rm dom}\, g$. Then, $h$ is twice epi-differentiable at $\bar u$, and $\partial h$ is proto-differentiable at  $\bar{u}$ for every $\bar v \in \partial h(\bar{u})$.
\end{lemma}

\begin{proof}
The twice epi-differentibility of $h$ at $\bar u$ for $\bar v\in \partial h(\bar u)$ follows from \cite[Proposition~3.1]{TW24}. By \cite[Theorem 2.2]{R 90} and convexity of $h$, this is equivalent to proto-differentiability of $\partial h$.
\end{proof}

The following result from  sheds light onto the relationship between the subspace parallel to the subdifferential of the the conjugate of a convex function and the kernel of its generalized Hessian.


\begin{theorem}[The kernel of the generalized Hessian of convex functions]\label{thm:c2conepargenhess}
    Let $g:\R^n\to \R\cup\{\infty\}$ be a proper, lsc, convex function and $(\bar{x}, \bar{v}) \in \operatorname{gph} \partial g$.
    We have \begin{align}\label{eq:pargenhess}
        \para \partial g^* (\bar{v}) \subseteq \ker \partial^2 g(\bar{x} | \bar{v}) \cup (- \ker \partial^2 g(\bar{x} | \bar{v})).
    \end{align}
     If $g^*$ is twice epi-differentiable at $\bar{v}$, then we have 
    \begin{align}\label{eq:pargenhesstw}
        \para \partial g^* (\bar{v}) \subseteq \ker \partial^2 g(\bar{x} | \bar{v}).
    \end{align}
    Moreover, if $g^*$ is $\mathcal{C}^2$-cone reducible at $\bar{v}$, the above inclusion turns to equality
    \begin{align}\label{eq:pargenhessc2}
        \para \partial g^* (\bar{v}) = \ker \partial^2 g(\bar{x} | \bar{v}).
    \end{align} 
\end{theorem}
\begin{proof}
 To justify \eqref{eq:pargenhess}, pick any $w \in \operatorname{par} \partial g^* (\bar{v}) \setminus \{0\}$
    and define $L := \operatorname{span}\{w\}$. Let $A$ be an $n\times n$ matrix with $\ker A = L$.
    Define $f(x) := \frac{1}{2} \lVert Ax - A \bar{x} \rVert^2 - \bar{v}^T x$.
    Note that $\nabla f(\bar{x}) = - \bar{v}$ and \begin{align*}
        \ker \nabla^2 f(\x) = \ker A^TA = \ker A = L.
     \end{align*} It follows that 
     \begin{align*}
        \ker \nabla^2 f(\x) \cap \operatorname{par} \partial g^* (\bar{v}) = L \not = \{0\}.
     \end{align*}
     By ~\Cref{thm:tiltstab2}, $\x$ is not a tilt-stable minimizer of $f + g$
    and \[
    \ker \nabla^2 f(\x) \cap \ker \partial^2 g(\x | \bar{v}) \not = \{0\}.
    \]
    Hence, there exists $\alpha \in \R \setminus \{0\}$ such that $\alpha w \in \ker \partial^2 g (\x | \bar{v})$, i.e., 
    $(0, - \alpha w) \in N_{\operatorname{gph} \partial g}(\bar{x}, \bar{v})$.
    As $N_{\operatorname{gph} \partial g}(\bar{x}, \bar{v})$ is a (not necessarily convex) cone and $\alpha \not = 0$,
    we obtain that \begin{align*}
        (0, w) \in N_{\operatorname{gph} \partial g}(\bar{x}, \bar{v}) \text{ or } (0, -w) \in N_{\operatorname{gph} \partial g}(\bar{x}, \bar{v}),
    \end{align*} 
    which implies that $w\in \ker \partial^2 g(\bar{x} | \bar{v}) \cup (- \ker \partial^2 g(\bar{x} | \bar{v}))$ for any $w\in\para \partial g^* (\bar{v}) \setminus \{0\}$. This verifies 
    the inclusion \eqref{eq:pargenhess}.

    To prove the inclusion ``$\subseteq$''  in \eqref{eq:pargenhesstw}, suppose further $g^*$ is  twice  epi-differentiable
    at $\bar{v}$.    By the convexity and closedness of $\partial g^* (\bar{v})$, we have that \cite[Theorem~6.3]{R 70}
    \begin{align*}
        \bar{x} \in \partial g^* (\bar{v}) = \operatorname{cl}(\operatorname{ri} \partial g^* (\bar{v})).
    \end{align*}
    Hence, there exists a sequence $\{x^k\} \subseteq \operatorname{ri} \partial g^*(\bar{v})$ such that $x^k \to \bar{x}$.
    We claim  that 
    \begin{align}\label{eq:coderivativeInclusion}
        T_{\partial g^* (\bar{v})}(x^k) = \operatorname{par} \partial g^* (\bar{v}) \subseteq \ker D \partial g( x^k | \bar{v}).
    \end{align} 
    The equality holds due to the fact that $x^k \in \operatorname{ri} \partial g^* (\bar{v})$ and \begin{align*}
        T_{\partial g^* (\bar{v})}(x^k) = \operatorname{cl}(\operatorname{cone}(\partial g^*(\bar{v}) - x^k)) = \operatorname{cl} (\operatorname{par} \partial g^* (\bar{v})) = \operatorname{par} \partial g^* (\bar{v}).
    \end{align*} To see the inclusion on the right-hand side of \eqref{eq:coderivativeInclusion},
    take $w \in T_{\partial g^* (\bar{v})}(x^k)$. There exist sequences $\{t_{k_n}\} \downarrow 0$ and \{$w^{k_n}\} \to w$
    such that $x^k + t_{k_n} w^{k_n} \in \partial g^* (\bar{v})$ for all $n$.
    It follows that \begin{align*}
       (x^k,\bar v)+t_{k_n}(w^{k_n},0)= (x^k + t_{k_n}w^{k_n}, \bar{v}) \in \operatorname{gph} \partial g,
    \end{align*} 
    which gives that $0 \in D \partial g (x^k | \bar{v})(w)$  or equivalently, $w \in \ker D \partial g^* (\bar{v} | x^k)$. Since $g^*$ is twice epi-differentiable at $\bar v$, $g$ is twice epi-differentiable at any $x_k$ for $\bar v$ by \cite[Theorem~13.21]{RoW 98}. The Rockafellar-Zagrodny derivative-coderivative inclusion (\cite[Theorem 13.57]{RoW 98}) tells us that
    \begin{align*}
        D \partial g (x^k | \bar{v}) \subseteq \partial^2 g(x^k | \bar{v}).
    \end{align*} This along with \eqref{eq:coderivativeInclusion} implies that, for all $k\in\mathbb N$, we have 
    \begin{align*}
        \operatorname{par} \partial g^* (\bar{v}) \subseteq \ker \partial^2 g(x^k | \bar{v})=\{w\in \R^n|\, (0,-w)\in N_{{\rm gph}\, \partial g}(x_k,\bar v)\}.
    \end{align*} 
    As the normal cone $N_{{\rm gph}\, \partial g}$ is a closed set-valued mapping,
    letting $x^k \to \bar{x}$ in the above inclusion gives us that  $\operatorname{par} \partial g^* (\bar{v}) \subseteq \ker \partial^2 g(\x| \bar{v})$,
    which shows the ``$\subseteq$" inclusion in \eqref{eq:pargenhesstw}.

    Next, let us suppose that $g^*$ is $\mathcal{C}^2$-cone reducible at $\bar{v}$.  By applying ~\Cref{lem:protodiff} for $h=g^*$, $g^*$ is twice epi-differentiable at $\bar{v}$. Hence, \eqref{eq:pargenhesstw} holds. We just need to justify the  inclusion ``$\supseteq$'' in \eqref{eq:pargenhessc2}. To do so, repeat some of the arguments from the beginning of the proof: 
    Take $z \in \ker \partial^2 g(\x | \bar{v}) \setminus \{0\}$ and define $L := \operatorname{span}\{z\}$.
    Let $A$ be an $n\times n$ matrix with $\ker A = L$. Define $f(x) := \frac{1}{2} \lVert Ax - A\bar{x} \rVert^2 -  \bar v^T x$.
    Again, we have $\nabla f(\x) = -\bar{v}$ and $\ker \nabla^2 f(\x) = L$. As $z \in L \cap \ker \partial^2 g(\x | \bar{v})$, $\x$ is
    not a tilt-stable minimizer of $f + g$. By ~\Cref{thm:tiltstab2}, $L \cap \operatorname{par} \partial g^*(\bar{v}) \not = \{0\}$. Thus, $z \in \operatorname{par} \partial g^* (\bar{v})$
    for every nonzero $z \in \ker \partial^2 g(\bar{x} | \bar{v})$.
\end{proof}

\noindent
The most important part for our study is the equality \eqref{eq:pargenhessc2} which furnishes a simple way of computing the generalized Hessian of a (closed, proper) convex function whose conjugate is $\mathcal{C}^2$-cone reducible and thus a tractable way of verifying the key condition  \eqref{eq:MordCrit} to deploy the implicit function theorem in \Cref{prop:Implicit} to study the solution map \eqref{eq:S}.

\section{Stability results for conjugate $\mathcal{C}^2$-cone reducible regularizers (and beyond)}\label{sec:Stability}

The following theorem extends the recent result \cite[Theorem~4]{CHNS 26}.

\begin{theorem}\label{prop:stabilityc2}
    Let $S : \R^{m \times n} \times  \R^m \times \R_{++} \rightrightarrows \R^n$ be the solution of problem \eqref{eq:generalLeastSquares} given by 
    \begin{align*}
        S(A, b, \lambda) := \underset{x \in \R^n}{\argmin} \left\{\frac{1}{2}\lVert Ax - b \rVert^2 + \lambda r(x)\right\}. 
    \end{align*}
Let $(\bar{A}, \bar{b}, \bar{\lambda}, \bar{x}) \in \operatorname{gph} S$ and
    $\bar{v} := -\frac{1}{\bar{\lambda}} \bar{A}^T(\bar{A} \bar{x} - \bar{b})$. Suppose that  $r^*$ is $\mathcal{C}^2$-cone reducible at $\bar v$. Then the following condition
    \begin{align}\label{eq:CQ}
        \ker \bar{A} \cap \para \partial r^* (\bar{v}) = \{0\}
    \end{align}
    is equivalent to any of the following conditions:
    \begin{itemize}
        \item[(a)] $S$ is single-valued at $(\bar A, \bar b, \bar \lambda)$.
        \item[(b)] $S$ is locally Lipschitz at $(\bar A, \bar b, \bar \lambda)$.
        \item[(c)] $S$ is directionally differentiable $(\bar A, \bar b, \bar \lambda)$.
    \end{itemize}
    Additionally, the following condition implies \eqref{eq:CQ}, and if
    $\partial r$ is semismooth* at $(\bar x, \bar v)$, then it is equivalent to \eqref{eq:CQ}:
    \begin{itemize}
        \item[(d)] $S$ is semismooth at $(\bar A, \bar b, \bar \lambda)$.
    \end{itemize}

\end{theorem}

\begin{proof}
    Apply \Cref{prop:Implicit} to $F = \partial r$, and $f$ defined by $f(p, x) = \frac{1}{\lambda} A^T(Ax - b)$, as was alluded to earlier.
 If condition \eqref{eq:CQ} is satisfied, by \eqref{eq:pargenhessc2} the Mordukhovich criterion \eqref{eq:MordCrit} that reads \begin{align*}
        \ker \bar{A} \cap \ker \partial^2 r(\bar{x} | \bar{v}) = \{0\}
    \end{align*} 
    holds. This gives that $S$ is (a) single-valued and (b) locally Lipschitz at $(\bar A, \bar b, \bar \lambda)$.
    By \Cref{lem:protodiff}, $\partial r^*$ is proto-differentiable at $\bar v$, so by \cite[Proposition 8.41]{RoW 98}, $\partial r$ is proto-differentiable at $\x$, therefore \Cref{prop:Implicit} also tells us that $S$ is (c) directionally differentiable at $(\bar A, \bar b, \bar \lambda)$.
    When $\partial r$ is semismooth* at $(\bar x, \bar v)$, \Cref{prop:Implicit} tells us that $S$ is semismooth* at $(\bar b, \bar \lambda)$. Thus, by \Cref{lem:Semi}, $S$ is semismooth at $(\bar b, \bar \lambda)$.

    On the other hand, (b), (c), and (d) all imply (a), which by
    \cite[Theorem 4]{CHNS 26}, is equivalent to \eqref{eq:CQ}.
\end{proof}


\begin{remark}[Conjugate $\mathcal{C}^2$-cone reducible functions]
    The class of proper, closed, convex functions with Fenchel conjugates which are $\mathcal{C}^2$-cone reducible includes polyhedral convex functions and support functions of $\mathcal{C}^2$-cone reducible sets (which encompasses the $\ell_1/\ell_2$ norm and the nuclear norm). See \cite[Remark 2]{CHNS 26} for a more comprehensive list.
\end{remark}

\subsection{Regularizers that are weighted polyhedral support functions}

\noindent
\Cref{prop:stabilityc2} is a solely qualitative statement. In order to obtain quantitative results regarding the stability of the solution map $S$, we need to impose more structure on the regularizer $r$ to arrive at computable expressions. In fact, we will first study the special case where the regularizer $r$ has the form  $r=\sigma_\mathcal P\circ M$, i.e., a support function of a polyhedral set composed with a linear map. This encapsulates, in particualar, all (weighted) polyhedral norms,  and therefore  stability results for the (weighted)  LASSO \cite{BBH 22, GKM 18}.  Functions of this form are polyhedral and are known to be $\mathcal{C}^2$-cone-reducible (\cite[Remark 3.6(a)]{CHNS 26}


 We start our study of this special case by a simple conjugacy result which appeals, in particular, to the fact that the support and indicator of a polyehdral set are polydreal convex functions.

\begin{lemma}[Conjugate of $\sigma_\pol\circ M$]\label{lem:rsubdiff}
    Consider  $r = \sigma_{\pol} \circ M$ where $\pol\subseteq \R^{\ell}$ is a polyhedron and $M \in \R^{\ell \times n}$. 
 Let 
$(\x, \bar v) \in \gph \partial r$ and set 
\begin{equation}\label{eq:Lambda}
  \Lambda(\x, \bar v) := \{y\in \R^{\ell} \ | \ M^Ty = \bar v , \ y \in N_{\pol}^{-1}(M\x)\}.  
\end{equation}
Then
\begin{align*}
    \partial r^* (\bar v) = M^{-1} (N_{\pol}(\bar y)), \ \forall \bar y \in \Lambda (\x, \bar v).
\end{align*} 
In particular, we have
\begin{align}\label{eq:rpara}
\para \partial r^*(\bar v) = \operatorname{span} M^{-1}( N_{\pol}(\bar y)) , \ \forall \bar y \in \Lambda(\x, \bar v).
\end{align}
\end{lemma}

\begin{proof}
    Using the relations $\partial r^* = (\partial r)^{-1}$  and $\partial \sigma_\pol=N_\pol^{-1}$ as well as the subdifferential chain rule \cite[Theorem 23.9]{R 70} with the fact that $\sigma_\pol$ is polyhedral, we find: 
    \begin{align*}
        u \in \partial r^*(\bar v) \iff \bar v \in M^{T} N_{\pol}^{-1}(Mu) \iff \exists y \in N_{\pol}^{-1}(Mu):\; v = M^Ty
    \end{align*} So, if $\bar y \in \Lambda(\x, \bar v)$, then \begin{align*}
        \partial r^*(\bar v) = M^{-1}(N_{\pol}(\bar y)).
    \end{align*} 
    And because $0 \in \partial r^*(\bar v)$, we have
    \begin{align*}
        \para \partial r^*(\bar v) = \lin M^{-1}(N_{\pol}(\bar y))\quad\mbox{for any}\quad \bar y\in \Lambda(\bar x,\bar v).
    \end{align*}

\end{proof}

\noindent
The next example shows that, in general, the $\lin$ cannot be moved inside in \eqref{eq:rpara}, and at the same time, that the subspace parallel to the subdifferential of the conjugate is readily available.

\begin{example}
    Let $M = \begin{pmatrix}
        1 & 1 \\ 1 & 1
    \end{pmatrix}$, $\pol = [0, 1]^2$, $\x = 0$, $\bar v = \begin{pmatrix}
        1 \\ 1
    \end{pmatrix}$, and $\bar y = \begin{pmatrix}
        1 \\ 0
    \end{pmatrix}$.
    First, we verify that $\bar y \in \Lambda (\bar x, \bar v)$.
    $M^T \bar y = \begin{pmatrix}
        1 \\ 1
    \end{pmatrix} = \bar v$, and $M \x = 0$, so $N_{\pol}^{-1}(M \x) = \R^2 \ni \bar y$. Now, we compute $\para \partial r^*(\bar v)$: $N_{\pol}(\bar y) = N_{[0,1]}(1) \times N_{[0, 1]}(0) = \R_{+} \times \R_{-}$, by \cite[Theorem 6.10]{RoW 98}.
    Thus, \begin{align*}
        M^{-1}N_{\pol}(\bar{y}) = \left\{z \in \R^2 \ | \ \begin{pmatrix}
        z_1 + z_2 \\
        z_1 + z_2
    \end{pmatrix} \in \R_{+} \times \R_{-}\right\} = \{z \ | \ z_1 + z_2 = 0\} = \lin \left\{\begin{pmatrix}
        1 \\ -1
    \end{pmatrix}\right\}.
    \end{align*} This tells us that $\lin M^{-1} N_{\pol} (\bar y) = \lin \left\{\begin{pmatrix}
        1 \\ -1
    \end{pmatrix}\right\}$. However, if we move the $\lin$ inside, we obtain \begin{align*}
        \lin N_{\pol}(\bar y) = \R^2 \implies M^{-1}(\lin N_{\pol}(\bar y)) = \R^2 \neq \lin \left\{\begin{pmatrix}
        1 \\ -1
    \end{pmatrix}\right\}.
    \end{align*}
    \hfill $\diamond$
\end{example}

\noindent
 The next result is a technical lemma needed for the upcoming stability result.

\begin{lemma}\label{lem:Lambda}  Let $\{(x_k,v_k)\in  \operatorname{gph} \partial r\}\to (\bar x,\bar v)$, and let $\bar y \in \Lambda(\bar x,\bar v)$ in \eqref{eq:Lambda}. Then there exists $\{y_k\in \Lambda(x_k,v_k)\}\to\bar y$.
\end{lemma}
 \begin{proof}
 Observe that $\gph \Lambda$ is given by \begin{align*}
     \{(x, v, y) \ | \ M^Ty = v, \, Mx \in N_{\pol}(y)\}.
 \end{align*} Because $N_{\pol}$ is polyhedral, we have that $\gph \Lambda$ is polyhedral, and therefore the set-valued map $\Lambda^{-1}$ has a graph which is a polyhedral convex set.
 For $k \in \mathbb N$, define $y_k := P_{\Lambda(x_k, v_k)}(\bar y)$, the projection of $\bar y$ onto $\Lambda (x_k, v_k)$. By \cite[Example 9.47]{RoW 98}, the set-valued map $\Lambda^{-1} : \R^{\ell} \rightrightarrows \R^{2n}$ is globally metrically regular, so there exists $\kappa > 0$, which does not depend on $k$, such that we have
 \begin{align*}
 \lVert y_k - \bar y \rVert = d(\bar y, \Lambda (x_k, v_k))
 \le \kappa d((x_k, v_k), \Lambda^{-1}(\bar y)) \le \kappa \left\lVert 
 \begin{pmatrix}
     x_k \\
     v_k
 \end{pmatrix} - \begin{pmatrix}
     \bar x \\
     \bar v
 \end{pmatrix}\right\rVert \to 0
 \end{align*} so $y_k \to \bar y$.
\end{proof}

\begin{corollary}\label{cor:polyhedraldd}
    Let $S : \R^m \times \R_{++} \rightrightarrows \R^n$ be given by
    \begin{align*}
        S(b, \lambda) = \underset{x \in \R^n}{\argmin} \left\{ \frac{1}{2} \lVert Ax - b \rVert^2 + \lambda r(x) \right\},
    \end{align*} where $r = \sigma_{\pol} \circ M$ for $\pol \subseteq \R^{\ell}$ polyhedral and $M \in \R^{\ell \times n}$.
    Let $\Lambda(\x, \bar v)$ be defined as in \Cref{lem:rsubdiff}.
    If for any $\bar y \in \Lambda(\x, \bar v)$, we have
    \begin{align}\label{eq:MordCritpolyhedral}
        \ker A \cap \operatorname{span} M^{-1} (N_{\pol}(\bar{y})) = \{0\}
    \end{align} then it holds that
    \begin{itemize}
        \item[(a)] $S$ is semismooth (in particular, single-valued, locally Lipschitz, and directionally differentiable) at $(\bar{b}, \bar{\lambda})$, and for each direction $(q, \alpha)$, there exists a matrix $Q$ with orthonormal columns such that the directional derivative in direction $(q, \alpha)$ is given by 
        \begin{align*}
            S'(\bar{b}, \bar{\lambda} ; q, \alpha) = Q[(AQ)^T(AQ)]^{-1} (AQ)^T \left(q + \frac{\alpha}{\bar{\lambda}}(A \x - \bar{b})\right).
        \end{align*}
        \item[(b)] For any $\bar{y} \in \Lambda(\x, \bar v)$, the modulus of Lipschitz continuity of $S$ can be bounded by
        \begin{align*}
            L \le \frac{1}{\sigma_{\min}(A \overline{Q})^2}
            \left( \sigma_{\max}(A \overline{Q}) + \frac{1}{\bar{\lambda}} \lVert (A\overline{Q})^T(A \x - \bar{b}) \rVert \right),
        \end{align*} where $\overline{Q}$ is a matrix with columns which
        form an orthonormal basis of $\operatorname{span} M^{-1} N_{\pol}(\bar{y})$.
    \end{itemize}
\end{corollary}

\begin{proof} First, observe that 
    $r$ is polyhedral convex, so by \cite[Theorem 11.14]{RoW 98}, $r^*$ is also polyhedral convex, and is therefore $\mathcal{C}^2$-cone reducible.
    Also, polyhedral convex functions are in particular, piecewise linear-quadratic, so by \cite[Proposition 3]{FGH 22}, we have that $\partial r$ is semismooth* at every point.
    Lemma~\ref{lem:rsubdiff} tells us that $\para \partial r^*(\bar v) = \lin M^{-1} (N_{\pol} (\bar{y}))$, for any $\bar y \in \Lambda(\x, \bar v)$.
    It will be useful going forward to parameterize $\pol$ as $\{z \ | \ P^Tz \preceq \beta\}$ for $P \in \R^{\ell \times k}$, $\beta \in \R^{k}$, and
    we will use $p_i$ to refer to the $i$th column of $P$.
    Proposition~\ref{prop:Implicit} tells us that the directional derivative of $S$ in direction $(q, \alpha)$
    is given by the unique $w$ which satisfies \begin{align*}
        0 \in DG(\bar{b}, \bar{\lambda}, \x | 0)(q, \alpha, w)
    \end{align*} where $G(b, \lambda, x) = \frac{1}{\lambda} A^T(Ax - b) + \partial r(x)$.
    Observe that for any $\bar{y} \in \Lambda (\x, \bar{v})$, by \cite[Theorem 7.2]{MMS 22}, we have this is equivalent to \begin{align*}
        \frac{1}{\bar{\lambda}} A^T \left(q + \frac{\alpha}{\bar{\lambda}} (A \x - \bar{b}) - Aw\right)
        \in D \partial r (\bar{x} | \bar{v})(w) = M^T D(N_{\pol})^{-1}(M \x | \bar{y})(Mw).
    \end{align*}
    Now, observe that \begin{align*}
        z \in D(N_{\pol})^{-1}(M \x | \bar{y})(Mw) \iff Mw \in D N_{\pol} (\bar{y} | M \x)(z) = N_{\mathcal{K}}(z),
    \end{align*} 
    where $\mathcal{K} = T_{\pol}(\bar{y}) \cap \{M \x\}^{\bot}$ is the {\em critical cone}, by \cite[Example 13.44]{RoW 98}.
    Notice that \begin{align*}
        N_{\mathcal{K}}(z) = N_{T_{\pol}(\bar{y})}(z) + N_{\{M \x\}^{\bot}}(z) = 
        \operatorname{cone} \{p_i \ | \ i \in I(\bar{y}), \ p_i^Tz = 0\} + \operatorname{span}\{M \x\}
    \end{align*} where $I(\bar{y}) := \{i \ | \ p_i^T \bar{y} = \beta_i\}$.
    We claim that $M\x \in \operatorname{cone} \{p_i \ | \ i \in K\}$. To see this, let $K := \{i \in I(\bar{y}) \ | \ p_i^Tz = 0\}$. Set $C := \operatorname{cone} \{p_i \ | \ i \in K\} \subseteq N_{\pol}(\bar{y})$.
    We know that $M \x \in N_{\pol}(\bar{y})$, and we in fact claim that $M \x \in C$. 
We know $M\bar{x}$ can be written as $\sum_{i \in I(\bar{y})} \gamma_i p_i$ for $\gamma_i \ge 0$.
Using the fact that $z^T(M\x) = 0$ with the fact that $p_i^Tz < 0$ for $i \in I(\bar{y}) \setminus K$, (which come from $z \in T_{\pol}(\bar{y})$ and $z \in \{M \x\}^{\bot}$) we get
\begin{align*}
    0 = \sum\limits_{i \in I(\bar{y})} \gamma_i p_i^Tz = \sum\limits_{i \in K} \gamma_i \underbrace{p_i^Tz}_{ = 0} + \sum\limits_{i \in I(\bar{v}) \setminus K} \gamma_i \underbrace{p_i^Tz}_{< 0}.
\end{align*} So $\gamma_i = 0$ for $i \not \in K$. We can then write
\begin{align*}
    Mw \in C + \operatorname{span}\{M\bar{x}\}.
\end{align*}
    Set $E := \operatorname{span} M^{-1}(C)$ and notice that \begin{align*}
        E =\operatorname{span} M^{-1}(C) \subseteq M^{-1} \operatorname{span} C = M^{-1} (\rge P_{K}).
    \end{align*}  Since $M \x \in C$, we have $\x \in E$. Therefore, $Mw \in C$, so $w \in E$.
    Then, since $z \in \ker P_K^T = (\rge P_K)^{\bot}$, we have
    \begin{align}\label{eq:inEperp}
        M^Tz \in M^T ((\rge P_K)^{\bot}) \overset{\eqref{lem:Sub}}{=} \left(M^{-1}(\rge P_K)\right)^{\bot} \subseteq E^{\bot}.
    \end{align}
    Now, let $q_1,...,q_r$ be an orthonormal basis of $E$
    and set $Q = \begin{bmatrix}
        q_1 \cdots q_r
    \end{bmatrix}$. Then, $QQ^T$ is the orthogonal projection onto $E$, and consequently $QQ^Tw = w$.
    Setting $\bar{u} := \frac{1}{\bar{\lambda}} (A \x - \bar{b})$, we now infer that \begin{align*}
        0 \overset{\eqref{eq:inEperp}}{=} Q^T (\lambda M^Tz) = Q^T A^T (q + \alpha \bar{u} - Aw).
    \end{align*} Using $w \in E$, this implies \begin{align*}
        (AQ)^T AQQ^Tw = (AQ)^T(q + \alpha \bar{u}).
    \end{align*} Since $E \subseteq \operatorname{span} M^{-1} N_{\pol}(\bar{y})$, using \eqref{eq:MordCritpolyhedral} gives
    $\ker A \cap \rge Q = \{0\}$, so $(AQ)^T(AQ)$ is invertible.
    Thus, we obtain \begin{align*}
        w = Q[(AQ)^T(AQ)]^{-1} (AQ)^T (q + \alpha \bar{u}).
    \end{align*} In view of \eqref{prop:Implicit}, this establishes the directional derivative of $S$ at $(\bar{b}, \bar{\lambda})$ in direction $(q, \alpha)$ is given by
    \begin{align*}
        S'(\bar{b}, \bar{\lambda} ; q, \alpha) = Q[(AQ)^T(AQ)]^{-1} (AQ)^T (q + \alpha \bar{u}).
    \end{align*}
    We now argue \eqref{eq:MordCritpolyhedral} holds locally:

    to this end, take any sequence $(b_k,\lambda_k)\to (\bar b, \bar \lambda)$ with $x_k=S(b_k,\lambda_k)\to\bar x$ and $v_k:=\frac{1}{\lambda_k}A^T (Ax_k-b_k)\in \partial r(x_k)\to \bar v.$ Then for $\bar y\in \Lambda(\bar x,\bar v)$, by \Cref{lem:Lambda}, there exists $y_k\in \Lambda(x_k,v_k)\to \bar y$, and thus (for $k$ sufficiently large)
\[
I(y_k)\subseteq I(\bar y) 
\]
hence $\partial r^*(v_k)\subseteq\partial r^*(\bar v)$.
Consequently, 
\[
\para \partial r^*(v_k)\cap \ker A \subseteq\para \partial r^*(\bar v)\cap \ker A,
\]
for $k$ sufficiently large.  Therefore, assuming \eqref{eq:MordCritpolyhedral} at $(\bar b,\bar \lambda)$ yields that this property holds for all $(b,\lambda)$ sufficiently close.

 Hence, by reiterating the above reasoning for nearby points, $S$ is directionally differentiable at $(b,\lambda)$ sufficiently close to $(\bar b,\bar \lambda)$, and $S'((b,\lambda);(\cdot,\cdot))$ is, in particular, continuous.  Thus,   from \Cref{prop:Implicit}(c) we infer that 
\[
L:=\limsup_{(b,\lambda)\to (\bar b,\bar \lambda)} \max_{\|(q,\alpha)\|\leq 1}\|S'((b,\lambda);(q,\alpha))\|
\]
is a local Lipschitz bound for $S$ at $(\bar b,\bar \lambda)$. Now let $(b_k,\lambda_k)\to (\bar b,\bar \lambda)$ such that
\[
 \max_{\|(q,\alpha)\|\leq 1}\|S'((b_k,\lambda_k);(q,\alpha))\|\to L.    
\]
As $S'((b_k,\lambda_k);(\cdot,\cdot))$ is continuous (for all $k\in \mathbb{N}$), there exists  $\{(q_k,\alpha_k)\in \mathbb{B}\}\to (\bar q,\bar \alpha)\in \mathbb{B}$ such that 
\[
\|S'((b_k,\lambda_k);(q_k,\alpha_k))\|\to L.
\]
 Let $x_k\in S(b_k,\lambda_k)\to \bar x$, $v_k=\frac{1}{\lambda_k}A^T(Ax_k-b_k)\in \partial r(x_k)\to \bar v$. With \Cref{lem:Lambda}, we choose 
  $y^k\in \Lambda(x_k,v_k)\to \bar y$. Let 
 \[
 K_k\subseteq I(y_k)=\set{p_i}{\ip{p_i}{y_k}=\beta_i}.
 \]
 be the associated index set from earlier with $(b_k,\lambda_k, x_k,v_k, y_k, q_k, \alpha_k)$ (instead of $(\bar b,\bar \lambda, \bar x,\bar v,\bar y, \bar q, \bar \alpha)$).
 By finiteness, we may assume w.l.o.g. that $K_k= K\subseteq I(\bar y)$. And thus we can assume w.l.o.g. that for some subspace $E$ we have 
 \[
 E_k=\lin M^{-1}(\cone\set{p_i}{i\in K_k})\equiv E\subseteq\para\partial r^*(\bar v).
 \]
 for the associated subspace $E_k$.  Now let $q_1,\dots, q_t\in \R^n$ be an orthonormal basis of $E$ such that $QQ^T\in\R^{n\times n}$ with $Q=[q_1,\dots, q_t]$ is the orthogonal projection onto $E$. Since $E\subseteq\para \partial r^*(\bar v)$, we can pad $Q$ to a matrix $\bar Q:=[Q\; \hat Q]$ whose columns form an orthonormal basis of $\para \partial r^*(\bar v)$ such that $\bar Q \bar Q^T\in \R^{n\times n}$ is the orthogonal projection onto $\para \partial r^*(\bar v)$.
 
With our derivations above we therefore find that
 \begin{eqnarray*}
 \|S'((b_k,\lambda_k);(q_k,\alpha_k))\|\ & = & \left\|Q[(AQ)^T(AQ)]^{-1}(AQ)^T\left[q_k+\frac{\alpha_k}{\lambda_k}(Ax_k-b_k)\right]\right\|\\
 & \leq & \|[(AQ)^T(AQ)]^{-1}\|\cdot \left\|(AQ)^T\left[q_k+\frac{\alpha_k}{\lambda_k}(Ax_k-b_k)\right]\right\|\\
 & = & \frac{1}{\lambda_{\min}((AQ)^T(AQ))}\cdot \left\|(AQ)^T\left[q_k+\frac{\alpha_k}{\lambda_k}(Ax_k-b_k)\right]\right\|.
 \end{eqnarray*}
Now observe that 
\begin{eqnarray*}
    \lambda_{\min}((AQ)^T(AQ))& = & \min_{x\in \R^n\setminus\{0\}}\frac{(Qx)^TA^TAQx}{\|x\|^2}\\
    &=& \min_{y\in E\setminus\{0\}}\frac{y^TA^TAy}{\|y\|^2}\\
    & \geq  & \min_{y\in \para \partial r^*(\bar v)\setminus\{0\}}\frac{y^TA^TAy}{\|y\|^2}\\
    & = & \lambda_{\min}((A\bar Q)^T(A\bar Q)).
\end{eqnarray*}
Here the first identity uses, e.g., \cite[Theorem 4.2.6]{HoJ 13}, while the second one relies on the fact that $\set{Qx}{x\in \R^n\setminus\{0\}}=E\setminus\{0\}$, and that $\|Q^Ty\|=\|y\|$ for all $y\in E$. The inequality  uses that $E\subseteq\para \partial r^*(\bar v)$, and the last identity uses the arguments from the first one and that the columns of $\bar Q$ form an orthonormal basis of $ \para \partial r^*(\bar v)$. 
Using this bound we get 
\begin{align*}
        \|S'((b_k,\lambda_k);(q_k,\alpha_k))\|\ \le 
        \frac{1}{\lambda_{\min}((A \bar Q)^T(A \bar Q))}\cdot \left\|(AQ)^T\left[q_k+\frac{\alpha_k}{\lambda_k}(Ax_k-b_k)\right]\right\|
        \\ = \frac{1}{\sigma_{\min}(A \bar Q)^2} \cdot \left\|(AQ)^T\left[q_k+\frac{\alpha_k}{\lambda_k}(Ax_k-b_k)\right]\right\|.
\end{align*} Taking the limit in $k$, we get \begin{align*}
    L \le \frac{1}{\sigma_{\min}(A \bar Q)^2} \cdot \left\|(AQ)^T\left[\bar q+\frac{\bar \alpha}{\bar \lambda}(A \x- \bar b)\right]\right\| \\
    \le \frac{1}{\sigma_{\min}(A \bar Q)^2} \left( \max_{q \in \B} \lVert (AQ)^T q \rVert + \max_{\alpha \in [-1, 1]} \lVert \frac{1}{\bar \lambda} (AQ)^T \alpha (A \x - \bar{b}) \rVert \right) \\
    = \frac{1}{\sigma_{\min}(A \bar Q)^2} \left( \lVert (AQ)^T \rVert + \lVert \frac{1}{\bar \lambda} (AQ)^T (A \x - \bar{b}) \rVert \right) \\
    = \frac{1}{\sigma_{\min}(A \bar Q)^2} \left( \sigma_{\max}(AQ) + \frac{1}{\bar \lambda} \lVert (AQ)^T (A \x - \bar{b}) \rVert \right) .
\end{align*} By a similar argument to the one above, we have \begin{align*}
    \lambda_{\max}((AQ)^T AQ) = \max_{x \in \R^n \setminus \{0\}} \frac{x^TQ^TA^TAQx}{\lVert x \rVert^2}
    = \max_{y \in E \setminus \{0\}} \frac{y^TA^TAy}{\lVert y \rVert^2}
    \le \lambda_{\max}((A \bar Q)^T A \bar Q)
\end{align*} which implies that $\sigma_{\max}(AQ) \le \sigma_{\max}(A \bar Q)$.
This gives a further bound of \begin{align*}
    L \le \frac{1}{\sigma_{\min}(A \bar Q)^2} \left( \sigma_{\max}(A \bar Q) + \frac{1}{\bar \lambda} \lVert (AQ)^T (A \x - \bar{b}) \rVert \right) .
\end{align*} Finally, because $\bar Q^T = \begin{bmatrix}
    Q^T \\
    \hat{Q}^T
\end{bmatrix}$, we have that \begin{align*}
    \frac{1}{\sigma_{\min}(A \bar Q)^2} \left( \sigma_{\max}(A \bar Q) + \frac{1}{\bar \lambda} \lVert (AQ)^T (A \x - \bar{b}) \rVert \right) \\ \le
    \frac{1}{\sigma_{\min}(A \bar Q)^2} \left( \sigma_{\max}(A \bar Q) + \frac{1}{\bar \lambda} \lVert (A \bar Q)^T (A \x - \bar{b}) \rVert \right).
\end{align*} which gives \begin{align*}
    L \le     \frac{1}{\sigma_{\min}(A \bar Q)^2} \left( \sigma_{\max}(A \bar Q) + \frac{1}{\bar \lambda} \lVert (A \bar Q)^T (A \x - \bar{b}) \rVert \right).
\end{align*}
\end{proof}

\begin{remark}
The above result could have been stated for $p = (A, b, \lambda)$ as opposed to $(b, \lambda)$, but due to space and legibility constraints, we have stated a simplified version. We could have also stated both \Cref{prop:stabilityc2} and \Cref{cor:polyhedraldd} for the more general setting where the data fidelity term is given by a generic twice continuously differentiable convex function $\varphi$, instead of $\frac{1}{2} \lVert \cdot \rVert^2$. Under this setting, the Mordukhovich criterion \eqref{eq:CQ} instead reads \begin{align*}
    \ker \bar{A}^T \nabla^2 \varphi(\bar A \bar x - \bar b) \bar{A} \cap \para \partial r^*(\bar v) = \{0\}.
\end{align*}
\end{remark}

\begin{example}[LASSO]
A popular choice of regularizer which can be written in the form $\sigma_{\pol} \circ M$ is the $\ell_1$ norm, in which case \eqref{eq:generalLeastSquares} becomes the celebrated LASSO problem \cite{T 96}. Stability results for the solution map of the LASSO problem can be found in \cite{BBH 22}. In order to apply \Cref{cor:polyhedraldd}, we take $\pol = \{z \, | \, P^Tz \preceq \beta \}$ for $P = \begin{bmatrix}
    I & -I
\end{bmatrix}$, $\beta$ equal to the vector with every component equal to one (which makes $\pol$ the unit ball under the $\ell_{\infty}$ norm), and $M = I$. Let $(\bar b, \bar \lambda, \bar x) \in \gph S$ and set $\bar v = -\frac{1}{\bar \lambda} A^T(A \bar x - \bar b)$. First, notice that when $M = I$, the set $\Lambda(\bar x, \bar v)$ defined in \Cref{lem:rsubdiff} reduces to $\{\bar v\}$. By \cite[Theorem 6.46]{RoW 98}, $N_{\pol}(\bar v) = \cone \{p_i \ | \ i \in J(\bar v)\}$ for $J(\bar v) = \{i \ | \ p_i^T\bar v = \beta_i\}$. From our choice of $P$ and $\beta$, this further simplifies to $N_{\pol}(\bar v) = \cone \{e_i \ | \ |\bar{v}_i| = 1\}$. Putting these facts together, we get that the Mordukhovich criterion \eqref{eq:MordCritpolyhedral} becomes \begin{align*}
    \ker A \cap \lin \{e_i \ | \ i \in J(\bar{v})\}=\{0\}.
\end{align*} This written as simply \begin{align*}
    \ker A_{J(\bar v)} = \{0\}.
\end{align*} When this holds, we may apply \Cref{cor:polyhedraldd} and obtain that $S$ is (single-valued) directionally differentiable and locally Lipschitz at $(\bar b, \bar \lambda)$.
Furthermore, by noticing that $\{e_i \ | \ i \in J(\bar v) \}$ gives an orthonormal basis of $\lin N_{\pol}(\bar v)$, we can obtain that the directional derivative in direction $(q, \alpha)$ given by
\begin{align*}
  S'(\bar b, \bar \lambda ; q, \alpha) = I_K\left((A_K^TA_K)^{-1} A_K^T (q + \frac{\alpha}{\bar \lambda}(A \x - \bar b)\right)
\end{align*} for some index set $K = K(q, \alpha) \subseteq J(\bar v)$ and the Lipschitz modulus given by \begin{align*}
    L \le \frac{1}{\sigma_{\min}(A_J)^2} \left(\sigma_{\max}(A_J) + \frac{1}{\bar \lambda} \lVert A_J^T (A \bar x - \bar b) \rVert\right).
\end{align*} Note that this agrees with the result in \cite[Theorem 4.13(a)]{BBH 22}.
\hfill $\diamond$
\end{example}

\subsection{PLQ penalties as  regularizers}\label{sec:PLQ}

As discussed in ~\Cref{lem:protodiff}, a $\mathcal{C}^2$-cone reducible function is twice epi-differentiable. However, a twice epi-differentiable may be not $\mathcal{C}^2$-cone reducible. In particular,  convex piecewise linear-quadratic functions (in the sense of \cite[Definition 10.20]{RoW 98}), which are twice epi-differentiable, are not necessarily $\mathcal{C}^2$-cone reducible,
as shown in the following example.

\begin{example} 
    Let $r^* : \R^2 \to \R$ be given by \begin{align*}
        r^*(v_1, v_2) = v_1^2 + v_2^2 + |v_1v_2|=\begin{cases} v_1^2 + v_2^2 + v_1v_2, & \text{if }v_1,v_2\ge 0,\\
        v_1^2 + v_2^2 - v_1v_2 & \text{if }v_1\le0,v_2\ge 0,\\
        v_1^2 + v_2^2 + v_1v_2 & \text{if }v_1\le0,v_2\le 0,\\
        v_1^2 + v_2^2 - v_1v_2 & \text{if }v_1\ge0,v_2\le 0.\\
        \end{cases}
    \end{align*}
Thus, $r^*$ is a convex piecewise linear-quadratic function; so is $r$ by \cite[Theorem 11.14]{RoW 98}.
Set $\bar{v} = (0, 0)$ and note that
    \begin{align*}
    \partial r^*(v_1, v_2) = \begin{cases}
    (2v_1 + v_2, 2v_2 + v_1) & \text{if } v_1v_2 > 0, \\
    (2v_1 - v_2, 2v_2 - v_1) & \text{if } v_1v_2 < 0, \\
    (2v_1, 2v_2) + [-|v_2|, |v_2|] \times \{0\} & \text{if } v_1 = 0, v_2 \not = 0, \\
    (2v_1, 2v_2) + \{0\} \times [-|v_1|, |v_1|]  & \text{if } v_2 = 0, v_1 \not = 0. \\
    \end{cases}
\end{align*}
This also implies that $\partial r^* (0,0) = \{(0,0)\}$. Set $\bar{x} = (0,0)$. Using the formula for the kernel of the generalized Hessian as a union of parallel subspaces (see \cite[Theorem 4.1]{RM 11}), 
we get \begin{align}
    \ker \partial^2 r(\bar{x}, \bar{v}) = \R \times \{0\} \cup \{0\} \times \R \not = \para \partial r^* (\bar{v}) = \{0\}.
\end{align}
Thus, $r^*$ cannot be $\mathcal{C}^2$-cone reducible at $\bar{v}$ by ~\Cref{thm:c2conepargenhess}.
    \hfill $\diamond$
\end{example}

\noindent
However, in some cases of $r$ being piecewise linear-quadratic, we still have 
the relation that the kernel of the generalized Hessian of $r$ is equal to the subspace parallel to $\partial r^*$.

\begin{definition}[Piecewise Linear-Quadratic Penalty]
    Let $\pol \subseteq \R^{n}$ be nonempty and polyhedral and let $B \in \R^{n \times n}$ be symmetric positive semidefinite.
    The function $\tpb$ defined by \begin{align*}
        \tpb(x) = \sup_{z \in \pol} \left\{x^Tz - \frac{1}{2} z^TBz \right\}
    \end{align*} is called a \emph{piecewise linear-quadratic penalty}.
    Note that $\tpb$ is proper, closed, and convex, and $\tpb^*(y) = \delta_{\pol}(y) + \frac{1}{2} y^TBy$ for $y\in \R^n$.
\end{definition}

As $\tpb^*(y) = \delta_{\pol}(y) + \frac{1}{2} y^TBy$, its epigraph is computed by
\[
\operatorname{epi} \tpb^* =h^{-1}(\pol\times\R_-)\;\; \mbox{with}\;\; h(y,r):=\left(y,\frac{1}{2}y^TBy-r\right).
\]
Since  $D h(y,r)=\begin{pmatrix} I& By\\
0&-1\end{pmatrix}$ is surjective for any $(y,r)\in \R^n\times \R$ and the polyhedral set $\pol\times\R_-$ is  $\mathcal{C}^2$-cone reducible, $\operatorname{epi} \tpb^*$ is also a $\mathcal{C}^2$-cone reducible set by \cite[Proposition~3.2]{Sh03}, i.e., $\tpb^*$ is a $\mathcal{C}^2$-cone reducible function.

In what follows, $r = \tpb \circ M$ where $M \in \R^{\ell \times n}$ and $\tpb$ is a piecewise linear-quadratic penalty,
parameterized by a positive semidefinite matrix $B \in \mathbb{S}^{\ell}_{+}$ and a polyhedral set $\pol \subseteq \R^{\ell}$. Although  $\tpb^*$ is $\mathcal{C}^2$-cone reducible, it is not presently clear to us if $r^*$ is $\mathcal{C}^2$-cone reducible. However, in ~\Cref{eq:kerHessPLQ} below, we show that the kernel of the generalized Hessian of the function $r$ enjoys the formula \eqref{eq:pargenhessc2} under a mild condition \eqref{eq:compactnessCQ} in the next result. 



\begin{lemma}[Compactness constraint qualification] \label{CCQ}
    Let $r = \tpb \circ M$ and suppose the following condition holds:
    \begin{align}\label{eq:compactnessCQ}
        \mathcal{P}^{\infty} \cap \ker B \cap \ker M^T = \{0\}
    \end{align} Then, for all $\bar{v} \in \rge M^T$, the following set is
    nonempty and compact:
    \begin{align}\label{eq:Tv}
        T(v) := \argmin \{q_B(y) + \delta_{\pol}(y) \ | \ M^Ty = v\} = \{y \ | \ r^*(v) = \tpb^*(y)\}
    \end{align}
\end{lemma}

\begin{proof}
    Let $v \in \rge M^T$. First, $T(v)$ may be written as the $\argmin$ of an unconstrained problem as follows:\begin{align*}
        T(v) = \underset{\mathbb{Y}}{\argmin} \, \varphi_v
    \end{align*} for $\varphi_v(y) = q_B(y) + \delta_{\pol}(y) + \delta_{(M^T)^{-1}\{v\}}(y)$.
    $\varphi_v$ is proper, closed, and convex, so by \cite[Proposition 3.1.3]{AT 03}, 
    nonemptiness and compactness of $\argmin \varphi_v$ is equivalent to positivity of the horizon function $\varphi_v^{\infty}(d)$ for all nonzero $d$.
    The horizon function of $\varphi_v$ can be computed to be \begin{align*}
        \varphi_v^{\infty}(d) = \delta_{\ker B}(d) + \delta_{\pol^{\infty}}(d) + \delta_{\ker M^T}(d)
        = \delta_{\ker B \cap \ker M^T \cap \pol^{\infty}} (d),
    \end{align*} which is equal to $0$ if and only if $d \in \ker B \cap \ker M^T \cap \pol^{\infty}$.
\end{proof}

The condition~\eqref{eq:compactnessCQ} is very mild. In particular, it holds when either $\mathcal{P}$ is a compact polyhedral set or $M$ is a surjective map.

The following result shows that when $r$ is a piecewise linear-quadratic penalty composed with a linear operator, we still have the aforementioned relation.

\begin{theorem}\label{eq:kerHessPLQ}
    Let $r = \tpb \circ M$. Let $(\bar{x}, \bar{v}) \in \gph \partial r$. Suppose \eqref{eq:compactnessCQ} holds. Then
    \begin{align*}
        \ker \partial^2 r(\bar{x} | \bar{v}) = \para \partial r^* (\bar{v}).
    \end{align*}
\end{theorem}

\begin{proof}
    According to the proof of \cite[Theorem 4.1]{RM 11}, for a sufficiently small
    neighborhood $\mathcal{O}$ of $\bar{v}$, we have \begin{align*}
        \partial^2 r^* (\bar{v} | \bar{x})(0) = \bigcup\limits_{v \in \mathcal{O}} \para \partial r^*(v).
    \end{align*} Because $\partial^2 r^*(\bar{v} | \bar{x})(0)$ is a union of subspaces, we can write it as
    $\ker \partial^2 r(\bar{x} | \bar{v})$. By shrinking the neighborhood to $\mathcal{O} \cap \rge M^T$, we get
    \begin{align}\label{eq:parker}
        \para \partial r^*(\bar{v}) \subseteq \ker \partial^2 r(\bar{x} | \bar{v}) = \bigcup\limits_{v \in \mathcal{O} \cap \rge M^T} \para \partial r^* (v).
    \end{align} By \cite[Theorem 11.33]{RoW 98}, we have that $r^*(v)$ is given by \begin{align*}
        r^*(v) = (M^T \tpb^*)(v) = \inf\{ \tpb^*(y) \ | \ M^Ty = v\}.
    \end{align*} Recall the set $T(v)$ in \eqref{eq:Tv}. Note that  $T$ is nonempty and compact-valued for all $v \in \rge M^T$ by ~\Cref{CCQ}.
    We can write $T(v)$ as \begin{align*}
        \{y \ |\ M^Ty=v, 0 \in By + N_{\pol}(y) + \rge M^T\}.
    \end{align*}
    Observe that $T(v)$ is the intersection of polyhedra, so $T$ is polyhedral. By \cite[3D]{DoR 14}, $T$ is outer Lipschitz continuous relative to its domain $\rge M^T$.
    Thus, shrinking $\mathcal{O}$ if necessary, we can obtain a (uniform) constant $\kappa > 0$ such that 
    \begin{align}\label{eq:TVV}
        T(v) \subseteq T(\bar{v}) + \kappa \lVert v - \bar{v} \rVert \mathbb{B} \text{ for all } v \in \mathcal{O} \cap \rge M^T.
    \end{align} For such a $v$, \cite[Proposition 4.3]{NVV 25} guarantees that we may write \begin{align*}
        \para \partial r^* (v) = M^{-1}(\para(\partial \tpb^*(y) \cap \rge M))
    \end{align*} for every $y \in T(v)$. We also have \begin{align*}
        \para \partial \tpb^*(y) = \para (N_{\pol}(y) + By) = \operatorname{span} N_{\pol}(y).
    \end{align*} By polyhedrality of $\delta_{\pol}$, for every $y \in T(v)$, there exists a neighborhood $V_y$ of $y$ such that
    \begin{align*}
        N_{\pol}(y') \subseteq N_{\pol}(y) \text{ for every } y' \in V_y.
     \end{align*} By compactness of $T(\bar{v})$, there exist finitely many $y_1,...,y_{\ell} \in T(\bar{v})$ such that \begin{align*}
        T(\bar{v}) \subseteq \bigcup\limits_{i=1}^{\ell} V_{y_{i}} := V.
     \end{align*} 
     $V$ is an open set containing $T(\bar{v})$, so shrinking $\mathcal{O}$ further if necessary, we can obtain
     \begin{align*}
        T(\bar{v}) + \kappa \lVert v - \bar{v} \rVert \mathbb{B} \subseteq V \text{ for } v \in \mathcal{O} \cap \rge M^T.
     \end{align*} This together with \eqref{eq:TVV} tells us $T(v) \subseteq V$ for any $v \in \mathcal{O} \cap \rge M^T$,
     so for any $y \in T(v)$, there is some $i \in [\ell]$ such that $y \in N_{y_{i}}$.
     Thus, we have 
     \begin{align*}
        \para \partial r^*(v) = M^{-1} (\para(N_{\pol}(y) \cap \rge M)) \subseteq
        M^{-1} (\para(N_{\pol}(y_i) \cap \rge M))\\ = \para \partial r^* (\bar{v}).
     \end{align*}
     Combining this with \eqref{eq:parker} gives 
     \begin{align}\label{eq:kerr}
        \ker \partial^2 r(\bar{x} | \bar{v}) = \para \partial r^* (\bar{v}) = M^{-1}(\para(N_{\pol}(y) \cap \rge M))
     \end{align} for any $y \in T(\bar{v})$.
\end{proof}

Applying this to the least squares case, we obtain the following:

\begin{theorem}\label{thm:plqstab}
    Let $r = \tpb \circ M$ and suppose $r$ satisfies the compactness qualification \eqref{eq:compactnessCQ}.
    Let $\pol$ be parameterized as $\{z \ | \ P^Tz \preceq \beta\}$ for $P \in \R^{\ell \times k}$, $\beta \in \R^{k}$, and let $p_i$ refer to column $i$ of $P$.
    Let $S : \R^{m \times n} \times \R^m \times \R_{++} \rightrightarrows \R^n$ be given by
    \begin{align*}
        S(A, b,\lambda)= \underset{x \in \R^n}{\argmin} \left\{ \frac{1}{2} \lVert Ax-b \rVert^2 +\lambda r(x)
        \right\}.
    \end{align*} 
    Let $(\bar{A}, \bar{b}, \bar{\lambda}, \bar{x}) \in \gph S$ (i.e. $\bar{x}$ solves \eqref{eq:generalLeastSquares} with parameters $\bar{A}$, $\bar{b}$, and $\bar{\lambda}$ and regularizer $r$).
    Set $\bar{v} = -\frac{1}{\bar{\lambda}} \bar{A}^T(\bar{A}\bar{x} - \bar{b})$.
    The condition that for some $\bar y \in T(\bar v)$, we have
    \begin{align}\label{eq:CQPLQM}
        \ker \bar{A} \cap M^{-1}(\operatorname{span}\{p_i \ | \ i \in I(\bar{y})\}) = \{0\}
    \end{align} where $I(\bar{y}) = \{i \ | \ p_i^T \bar y  = \beta_i\}$, 
    is equivalent to any of the following conditions:
\begin{itemize}
        \item[(a)] $S$ is single-valued at $(\bar{A}, \bar{b}, \bar{\lambda})$.
        \item[(b)] $S$ is locally Lipschitz at $(\bar{A}, \bar{b}, \bar{\lambda})$.
        \item[(c)] $S$ is directionally differentiable at $(\bar{A}, \bar{b}, \bar{\lambda})$.
        \item[(d)] $S$ is semismooth at $(\bar{A}, \bar{b}, \bar{\lambda})$. 
\end{itemize}
\end{theorem}

\begin{proof}

Let $\bar{x} \in S(\bar{A}, \bar{b}, \bar{\lambda})$.
As was discussed in \Cref{sec:Kernel}, the condition to apply \Cref{prop:Implicit} to the regularized least squares setting is
\begin{align*}
    \ker \bar{A} \cap \ker \partial^2 r(\bar x | \bar v) = \{0\}.
\end{align*} By the previous result, this is equivalent to
\begin{align}\label{eq:PLQMpara}
    \ker \bar{A} \cap \para \partial r^*(\bar v) = \{0\}.
\end{align} As discussed in the proof of \Cref{eq:kerHessPLQ}, the object $\para \partial r^* (\bar v)$ can be written as $M^{-1}(\lin \np(\bar y)$ for any $\bar y \in T(\bar v)$. Also, by \cite[Theorem 6.46]{RoW 98}, $\np(\bar y) = \cone \{p_i \ | \ i \in I(\bar y)\}$, so putting these facts together, we obtain that the condition to apply \Cref{prop:Implicit} is \begin{align*}
    \ker \bar{A} \cap M^{-1}(\operatorname{span}\{p_i \ | \ i \in I(\bar{y})\}) = \{0\}.
\end{align*} This tells us that when this holds, $S$ is (a) single-valued and (b) locally Lipschitz at $(\bar A, \bar b, \bar \lambda)$, and because $r$ is piecewise linear-quadratic, by \cite[Remark 1]{FGH 22}, $\partial r$ is proto-differentiable at $(\x, \bar v)$, so $S$ is also (c) directionally differentiable at $(\bar A, \bar b, \bar \lambda)$. By \cite[Proposition 3]{FGH 22}, $\partial r$ is semismooth* at $(\bar x, \bar v)$, so $S$ is semismooth* at $(\bar A, \bar b, \bar \lambda)$. Combining this with (c), we get that $S$ is (d) semismooth at $(\bar A, \bar b, \bar \lambda)$. On the other hand, (b)-(d) all imply (a), which by \cite[Theorem 3]{CHNS 26} is equivalent to \eqref{eq:PLQMpara} which as established, is equivalent to \eqref{eq:CQPLQM}
\end{proof}


In the case that $M = I$, we have a closed-form 
expression for the directional derivative and Lipschitz constant of $S$:
\begin{corollary}\label{cor:Quant}
    Let $S : \R^m \times \R_{++} \rightrightarrows \R^n$ be given by \begin{align*}
        S(b, \lambda) = \underset{x \in \R^n}{\argmin} \left\{ \frac{1}{2} \lVert Ax - b \rVert^2 + \lambda \tpb(x) \right\}.
    \end{align*}
    Let $\bar x \in S(\bar b, \bar \lambda)$ and set $\bar v = - \frac{1}{\bar \lambda}A^T(A\x - \bar b)$. Note that under this setup, the assumption in \eqref{CCQ} is automatically satisfied. In this case, the Mordukhovich criterion simplifies to \begin{align}\label{eq:CQsimple2}
        \ker A \cap \operatorname{span}\{p_i \ | \ i \in I(\bar{v})\}
    \end{align}
    where $I(\bar{v}) = \{i \ | \ p_i^T \bar y  = \beta_i\}$. 
    Then, if \eqref{eq:CQsimple2}  holds, we have the following:
    \begin{itemize}
        \item[(i)] $S$ is directionally differentiable at $(\bar{b}, \bar{\lambda})$, and for every direction $(q, \alpha)$, there exists a matrix $U$ with orthonormal columns such that the directional derivative in direction $(q, \alpha)$ is given by
        \begin{align*}
            (H^{-1}BA^T + U[(AHU)^T (AU)]^{-1} (AHU)^{T} (I - AH^{-1}BA^T)) \left(q + \frac{\alpha}{\bar{\lambda}}(A \bar{x} - \bar{b})\right)
        \end{align*} for $H = I + \frac{1}{\bar \lambda} BA^TA$.
        \item[(ii)] $S$ is single-valued and locally Lipschitz around $(\bar{b}, \bar{\lambda})$, and the modulus of Lipschitz continuity can be bounded by
        {\small  \begin{align*}
            L \le \max_{U' \in \mathcal{V}} 
    \left\lVert (H^{-1}BA^T + U'[(AH U')^T (AU')]^{-1} (AH U')^{T} (I - AH^{-1}BA^T)) \left(\bar q - \frac{\bar \alpha}{\bar \lambda}(A \bar x - \bar b)\right) \right\rVert
        \end{align*}}
        for  $\mathcal{V}:=\{V\in \R^{n\times r} \ | \ V^TV=I, r\le |I(\bar v)|\}$.
    \end{itemize}
\end{corollary}

\begin{proof}
    By ~\Cref{prop:Implicit}, at a point $(\bar{b}, \bar{\lambda}, \bar{x}) \in \gph S$ such that \eqref{eq:CQsimple2} holds, we have that $S$ is directionally 
differentiable, ~\Cref{prop:Implicit} tells us that the directional derivative in direction $(q, \alpha)$ given by 
$S'(\bar{b}, \bar{\lambda}; q, \alpha) = w$ for the unique $w$ such that
\begin{align}\label{eq:0DG}
    0 \in D G(\bar{b}, \bar{\lambda}, \bar{x} | 0)(q, \alpha, w),
\end{align}
where $G(b, \lambda, x) = \frac{1}{\lambda} A^T(Ax - b) + \partial \tpb (x)$.
By ~\Cref{lem:Sum}(a) this is equivalent to \begin{align*}
    \frac{1}{\bar{\lambda}} A^T(q + \frac{\alpha}{\bar{\lambda}}(A \bar{x} - \bar{b}) - Aw) \in D \partial \tpb (\bar{x} | \bar{v})(w).
\end{align*} Set $z := \frac{1}{\bar{\lambda}} A^T(q + \frac{\alpha}{\bar{\lambda}}(A \bar{x} - \bar{b}) - Aw)$. Applying the inversion formula for graphical derivatives \cite[8(19)]{RoW 98}, this can be written as \begin{align*}
    w \in D(\partial \tpb^*)(\bar{v} | \bar{x})(z) = D(N_{\pol} + B)(\bar{v} | \bar{x})(z) = D N_{\pol} (\bar{v} | \bar{x} - B \bar{v})(z) + Bz \\
    \iff w - Bz \in D N_{\pol} (\bar{v} | \bar{x} - B \bar{v})(z).
\end{align*} Let $\mathcal{K} := T_{\pol}(\bar{v}) \cap \{\bar{x} - B\bar{v}\}^{\bot}$ being the {\em critical cone} of $\pol$ at $\bar v$ for $\bar{x} - B\bar{v}$.
By \cite[Example 13.44]{RoW 98}, we can write the graphical derivative of the normal cone as follows:
\begin{align*}
    w - Bz \in N_{\mathcal{K}}(z) = N_{T_{\pol}(\bar{v})}(z) + N_{\{\bar{x} - B \bar{v}\}^{\bot}}(z).
\end{align*} So $z \in T_{\pol}(\bar{v})$ and $z^T(\bar{x} - B \bar{v}) = 0$. These can be written as \begin{align*}
    &N_{T_{\pol}(v)}(z) = \operatorname{cone} \{p_i \ | \ i \in I(\bar{v}), p_i^Tz = 0\},\\
    &N_{\{\bar{x} - B \bar{v}\}^{\bot}}(z) = \operatorname{span}\{\bar{x} - B \bar{v}\}.
\end{align*} 
It follows that 
\[
w - Bz \in \operatorname{cone} \{p_i \ | \ i \in I(\bar{v}), \ p_i^Tz = 0\} + \operatorname{span}\{\bar{x} - B \bar{v}\}.
\]
Using a similar argument to the proof of Corollary~\ref{cor:polyhedraldd}, let $J := \{i \in I(\bar{v}) \ | \ p_i^Tz = 0\}$. The fact that $\bar{v} \in \partial \tpb (\bar{x}) = (N_{\pol} + B)^{-1}(\bar{x})$ gives us 
\[
\bar{x} - B \bar{v} \in N_{\pol}(\bar{v}) = \operatorname{cone}\{p_i \ | \ i \in I(\bar{v})\}.
\]
We claim that in fact $\bar{x} - B \bar{v} \in \operatorname{cone} \{p_i \ | \ i \in J\}$.
Indeed, as $\bar{x} - B \bar{v}$ can be written as $\sum_{i \in I(\bar{v})} \gamma_i p_i$ for $\gamma_i \ge 0$,  we have 
\begin{align*}
    0 = z^T(\bar x-B\bar{v})=\sum\limits_{i \in I(\bar{v})} \gamma_i p_i^Tz = \sum\limits_{i \in J} \gamma_i \underbrace{p_i^Tz}_{ = 0} + \sum\limits_{i \in I(\bar{v}) \setminus J} \gamma_i \underbrace{p_i^Tz}_{< 0}.
\end{align*} So $\gamma_i = 0$ for $i \not \in J$. This clarifies the claim. It follows that 
\begin{align*}
    w - Bz \in \operatorname{cone}\{p_i \ | \ i \in J\} + \operatorname{span}\{\bar{x} - B \bar{v}\} \subseteq \operatorname{span}\{p_i \ | \ i \in J\} + \operatorname{span}\{p_i \ | \ i \in J\} 
    \\ = \operatorname{span}\{p_i \ | \ i \in J\}.
\end{align*}

Let $E := \operatorname{span}\{p_i\ |\ i \in J\}$.
We have that $w - Bz \in E$ and $z \in E^{\bot}$. Let $\bar{u} := - \frac{1}{\bar{\lambda}}(A \bar{x} - \bar{b})$, $H := I + \frac{1}{\bar{\lambda}} BA^{T} A$, and $h := -\frac{1}{\bar{\lambda}} BA^{T}(q - \alpha \bar{u})$.
Recall that $z = \frac{1}{\bar{\lambda}} A^T(q + \frac{\alpha}{\bar{\lambda}}(A \bar{x} - \bar{b}) - Aw)$. Thus we can rewrite $w - Bz$ as follows: \begin{align*}
    &w - Bz = w - \frac{1}{\bar{\lambda}}BA^{T}(q - \alpha \bar{u} - Aw) = 
    \\ &w + \frac{1}{\bar{\lambda}} BA^{T} Aw - \frac{1}{\bar{\lambda}} BA^{T} (q - \alpha \bar{u}) = Hw + h
\end{align*}
Note that $H$ is invertible by ~\Cref{lem:Hinvert}.  As $Hw + h \in E$, we have  $w + H^{-1}h \in H^{-1}(E)$. Let $\begin{bmatrix}
    u_1 \hdots u_r
\end{bmatrix} = U$ be an orthonormal basis of $H^{-1}(E)$. Since $J\subset I(\bar v)$, the qualification condition \eqref{eq:CQsimple2}
implies that $\ker A \cap E = \{0\}$ with $r=|J|$, the cardinality of $J$. Moreover,  the fact that $z \in E^{\bot}$ gives us $H^Tz \in H^T(E^{\bot}) \overset{\eqref{lem:Sub}}{=} (H^{-1}(E))^{\bot}$, so $U^TH^Tz = 0$.
 We can rewrite this as $
    (AHU)^T(q - \alpha \bar{u} - Aw) = 0$,
 which gives 
 \begin{align}
     (AHU)^T Aw = (AHU)^T (q - \alpha \bar{u}).
 \end{align} Adding $(AHU)^T AH^{-1}h$ to both sides and using the fact $UU^T$ is the orthogonal projection onto $H^{-1}(E)$, we obtain
 \begin{align*}
    (AHU)^{T} (q - \alpha \bar{u} + AH^{-1}h) = (AHU)^{T} A UU^{T}\underbrace{(w + H^{-1}h)}_{\in H^{-1}(E)}.
 \end{align*}
 By Lemma~\ref{lem:AHUinvert}, $(AHU)^T AU$ is invertible. We can then write
 \begin{align*}
    & U^T(w + H^{-1}h) = [(AHU)^T (AU)]^{-1} (AHU)^{T} (q - \alpha \bar{u} + AH^{-1}h)
 \end{align*} and isolate $w$ as follows:
 \begin{align*}
    w = -H^{-1}h + U[(AHU)^T (AU)]^{-1} (AHU)^{T} (q - \alpha \bar{u} + AH^{-1}h).
 \end{align*}
 Plugging back in the definition of $h$, this yields
 \begin{align*}
    w = (H^{-1}BA^T + U[(AHU)^T (AU)]^{-1} (AHU)^{T} (I - AH^{-1}BA^T)) (q - \alpha \bar{u}),
 \end{align*} 
 which verifies the formula of directional derivative of $S$ at $(\bar b,\bar \lambda)$ at $(q,\alpha)$ in (i). 

To verify (ii), we first notice that \eqref{eq:CQsimple2} holds locally around $(\bar{b}, \bar{\lambda}, \bar{x})$. Indeed, 
 let $\{(b_k, \lambda_k, x_k)\} \subseteq \gph S$ be a sequence converging to $(\bar{b}, \bar{\lambda}, \bar{x})$.
 Set $v_k := - \frac{1}{\lambda_k} A^T(Ax_k - b_k)$. For $k$ sufficiently large, we have $I(v_k) \subseteq I(\bar{v})$, so
 for such $k$ we have \begin{align*}
    \ker A \cap \operatorname{span} \{p_i \ | \ i \in I(v_k)\} \subseteq \ker A \cap \operatorname{span} \{p_i \ | \ i \in I(\bar{v})\} = \{0\}.
 \end{align*}
 This establishes that for $(b, \lambda)$ sufficiently close to $(\bar{b}, \bar{\lambda})$, $S$ is 
 directionally differentiable. Because we eventually have the containment $I(v_k) \subseteq I(\bar{v})$,
 the formula for the directional derivative holds for $(b, \lambda)$ close enough to $(\bar{b}, \bar{\lambda})$ (if we replace $\bar{\lambda}$ for $\lambda$ and $b$ for $\bar{b}$).

 When \eqref{eq:CQsimple2} holds at $(\bar{b}, \bar{\lambda}, \bar{x}) \in \gph S$, we have that $S$ is locally Lipschitz 
 at $(\bar{b}, \bar{\lambda})$ with modulus \begin{align*}
    L \le \limsup_{(b, \lambda) \to (\bar{b}, \bar{\lambda})} \max_{\lVert (q, \alpha) \rVert \le 1} \lVert S'(b, \lambda ; q, \alpha) \rVert.
 \end{align*}
 Let $\{(b_k, \lambda_k)\}$ be a sequence converging to $(\bar b, \bar \lambda)$ such that 
 \begin{align*}
     L \le \lim_{k \to \infty} \max_{(q, \alpha) \in \B} \lVert S'(b_k, \lambda_k ; q, \alpha) \rVert. 
 \end{align*} We may also choose a corresponding sequence $\{(q_k, \alpha_k)\} \subseteq \B$ such that \begin{align*}
     L \le \lim_{k \to \infty} \lVert S'(b_k, \lambda_k ; q_k, \alpha_k) \rVert.
 \end{align*} 
 Let $\{x_k \in S(b_k, \lambda_k)\}$ such that $x_k \to \x$. For $k$ sufficiently large, we have that $(b_k, \lambda_k)$ are in a neighborhood of $(\bar b, \bar \lambda)$ such that \eqref{eq:CQsimple2} holds, so by the same arguments as in the proof for (i),  we can define $H_k := I + \frac{1}{\lambda_k}BA^TA$, and we have that for the direction $(q_k, \alpha_k)$ there exists a subspace $E_k$ and matrix with orthonormal columns $U_k\in \R^{n\times r_k}$ with $\rge U_k = H_k^{-1}(E_k)$ and $r_k\le|I(v_k)|\le |I(\bar v)|$ such that the directional derivative of $S$ at $(b_k, \lambda_k)$ in direction $(q_k, \alpha_k)$ is given by
 \begin{align*}
     &S'(b_k, \lambda_k ; q_k, \alpha_k) = \\
     &(H_k^{-1}BA^T + U_k[(AH_kU_k)^T AU_k]^{-1}(AH_kU_k)^T (I - AH_k^{-1}BA^T))(q_k - \frac{\alpha_k}{\lambda_k}(Ax_k - b_k)).
 \end{align*} 
 The sequence $\{(q_k, \alpha_k)\} \subseteq \B$ is bounded, so we can assume without loss of generality that it converges to some $(\bar q, \bar \alpha)$ and $r_k=r$.
 Also,  every $U_k\in \R^{n\times r}$ satisfies $U_k^TU_k = I$, so the sequence $\{U_k\}$ is bounded, and therefore we can also assume it converges to some $\bar U\in \R^{n\times r}$ such that $\bar U^T \bar U = I$. Then, passing to the limit and upper bounding by taking a maximum over the space of $n \times r$ matrices satisfying $V^TV = I$, we get
 {\small\begin{align*}
     &L \le \left\lVert (H^{-1}BA^T + \bar{U}[(AH \bar{U})^T (A \bar{U})]^{-1} (AH \bar{U})^{T} (I - AH^{-1}BA^T)) \left(\bar q - \frac{\bar \alpha}{\bar \lambda}(A \bar x - \bar b)\right) \right\rVert \\ & \le \max_{U' \in \mathcal{V}_{n,r}} 
    \left\lVert (H^{-1}BA^T + U'[(AH U')^T (AU')]^{-1} (AH U')^{T} (I - AH^{-1}BA^T)) \left(\bar q - \frac{\bar \alpha}{\bar \lambda}(A \bar x - \bar b)\right) \right\rVert
 \end{align*}}
 where $\mathcal{V}_{n,r} := \{V \in \R^{n \times r}\ | \ V^TV = I\}$, which is compact. Finally, we take a maximum over $\mathcal{V} := \{V \in \R^{n \times r} \ | \ V^TV = I, \ r \le |I(\bar v)|\}$ to obtain
 {\small
 \begin{align*}
  L \le
     \max_{U' \in \mathcal{V}} \left\lVert (H^{-1}BA^T + U'[(AH U')^T (AU')]^{-1} (AH U')^{T} (I - AH^{-1}BA^T)) \left(\bar q - \frac{\bar \alpha}{\bar \lambda}(A \bar x - \bar b)\right) \right\rVert.
 \end{align*}
}


\end{proof}

\begin{remark} Again, as with \Cref{cor:polyhedraldd}, we could have stated this result for the least squares problem parameterized by $(A, b, \lambda)$ instead of $(b, \lambda)$, but we have stated a simpler result due to space and legibility constraints. We could also have stated \Cref{thm:plqstab} for the setting where a (locally) twice continuously differentiable convex function $\varphi$ is used for data fidelity, in which case \eqref{eq:CQPLQM} reads
\begin{align*}
    \ker \bar{A}^T\nabla^2 \varphi(\bar{A} \x - \bar b)\bar{A} \cap M^{-1}(\lin \{p_i \ | \ i \in I(\bar y)\}) = \{0\}.
\end{align*}
\end{remark}


\section{Final Remarks}\label{sec:Final} 
In this paper, we studied stability properties of solution mappings associated with convex regularized least-squares problems through a variational-analytic framework.  Our approach combines an implicit-function perspective for monotone generalized equations with second-order tools from generalized differentiation, leading to conditions for local Lipschitz continuity, directional differentiability, and semismoothness of the associated solution maps.

A central aspect of the analysis was the relationship between generalized Hessians and the geometry of the subdifferential of the conjugate regularizer. In particular, for conjugate \(\mathcal{C}^2\)-cone reducible regularizers, we showed that the kernel of the generalized Hessian admits a tractable characterization in terms of the parallel subspace of the subdifferential. This allows one to replace difficult second-order calculations by more explicit first-order geometric conditions.

Our framework applies to a broad range of regularizers arising in optimization, statistics, and machine learning, including weighted polyhedral support functions and piecewise linear-quadratic penalties. In addition to qualitative stability results, we also derived quantitative sensitivity estimates and explicit formulas for directional derivatives in structured settings.

Several directions remain open for future research: 
\begin{itemize}
\item[i)] A natural question is to extend the family of regularizers that satisfy the generalized Hessian  expression \eqref{eq:pargenhessc2} beyond the ones given here so that the implicit function theorem from \Cref{sec:Implicit}
 can be readily applied.
 \item[ii)] The semismoothness results for the optimal solution maps in \Cref{sec:Stability} beg the question if explicit formulae for the Clarke Jacobians of said solution maps can be computed. This would be appealing for semismooth Newton methods. 

 \item[iii)] In \Cref{sec:PLQ}, we discuss that it is presently unclear whether the composition of a PLQ penalty with a linear map is  conjugate  $\mathcal{C}^2$-cone reducible. A positive answer would make the compactness \eqref{eq:compactnessCQ} redundant. 
 We would like to resolve this question in the future.  
 
\end{itemize}

\section*{Acknowledgments} The first author would like to thank Ebrahim Sarabi (Miami Unversity) for fruitful discussions about material related to this work.

\end{document}